\documentclass[11pt, a4 paper]{amsart} 
\usepackage{pdfsync}
\usepackage[psamsfonts]{amssymb} 
\usepackage{amsfonts,  amsmath,  mathabx}
\usepackage{graphicx}
\usepackage{enumerate,  enumitem,  amscd}
\usepackage{amsthm,  amssymb,  epsfig,  psfrag,  hyperref,  amsmath}

\usepackage[arrow,curve,matrix,tips,2cell]{xy} \SelectTips{eu}{10}
\UseTips \UseAllTwocells
  
\DeclareMathAlphabet{\matheurm}{U}{eur}{m}{n}

\newtheorem{thmA}{Theorem}

\numberwithin{equation}{section} 

\newtheorem{theorem}{Theorem}[section]

\newtheorem{proposition}[theorem]{Proposition}
\newtheorem{lemma}[theorem]{Lemma}

\newtheorem{question}[theorem]{Question}
\newtheorem{conjecture}[theorem]{Conjecture}   
\newtheorem*{basic-question}{Basic Question}

\theoremstyle{remark}
\newtheorem{remark}[theorem]{Remark}

\newtheorem{example}[theorem]{Example}

\theoremstyle{definition}
\newtheorem{definition}[theorem]{Definition}

\def\Z{\mathbb Z}
\def\H{\mathbb H}
\def\C{\mathbb C}

\def\aut{{\rm{Aut}}}

\def\norm{\unlhd}
\def\tr{{\rm{tr}}}

\def\PSL{{\rm{PSL}}}

\def\D{\Delta}  
\def\Z{\mathbb Z}

\def\R{\mathbb R}
\def\Q{\mathbb Q}

\def\G{\Gamma}

\def\La{\Lambda}

\def\ssm{\smallsetminus}

\def\e{\varepsilon}
\def\SL{{\rm{SL}}}
\def\-{\overline}
\def\wh{\widehat}

\def\onto{\twoheadrightarrow}

\def\G{\Gamma}

\def\<{\langle}
\def\>{\rangle}

\begin{document}

\title{Chasing finite shadows of infinite groups through geometry}

\author[Martin R.   Bridson]{Martin R.  ~Bridson}
\address{Mathematical Institute\\
Andrew Wiles Building\\
Woodstock Road\\
Oxford,   OX2 6GG}
\email{bridson@maths.ox.ac.uk}

\date{Tuesday 8 April 2025, Stanford} 



\thanks{For the purpose of open access, the author has applied a CC BY public copyright licence to any author accepted manuscript arising from this submission.}

\begin{abstract}   There are many situations in geometry and group theory where it is natural,   convenient or necessary to explore infinite groups via their actions on finite objects – i.e.   via the finite quotients of the group.   But how much understanding can one really gain about an infinite group by examining its finite images? Which properties of the group can one recognise,   and when does the set of finite images determine the group completely? How hard is it to decide what the finite images of a given infinite group are?
 
These notes follow my plenary lecture at the ECM in Sevilla,   July 2024.   The goal of the lecture was 
to sketch some of the rich history of the preceding problems and to present results that illustrate how the field
surrounding these questions has been transformed in recent years by input from low-dimensional topology and the study of non-positively curved spaces.  
\end{abstract}

\maketitle

\def\<{\langle}
\def\>{\rangle}
\def\La{\Lambda}
\def\onto{\twoheadrightarrow}
\def\G{\Gamma}
\def\Free{{\rm{Free}}}
\def\e{\varepsilon}

\newcommand{\magenta}[1]{{\color{magenta} #1}}

\newcommand{\co}{\colon \thinspace}    
\newcommand{\IP}{{\rm{IP}}} 
\newcommand{\IPk}{{\rm{IP}}^{(k)}}  
\newcommand{\FillM}{{\rm{Fill}_0^M}}
\newcommand{\fa}{{ {magenta}{\rm{FArea}}}} 
\newtheorem{conjectures}[theorem]{Conjectures} 
 
\def\a{\alpha}
\def\b{\beta}
\def\c{\gamma}
\def\d{\delta}

\newtheorem{vision}[theorem]{Bill Thurston's Vision}

\numberwithin{equation}{section}


\newcommand{\CG}{{\mathcal{C}\mathcal{G}}}

\def\id{{\rm{id}}}
\def\dom{{\rm{dom}}}
\def\ran{{\rm{ran}}}

\newcommand{\xycomsquare}[8]                   
{\xymatrix{#1 \ar[r]^{#2} \ar[d]^{#4} &
    #3 \ar[d]^{#5}  \\
    #6\ar[r]^{#7} & #8 }}

 \section{Introduction}

{\em To what extent can one understand an infinite group by studying only its finite images?} 
  This compelling question has re-emerged with different emphases throughout the history of group theory,   and in recent years it has been animated by a rich interplay with geometry and low-dimensional topology.    My purpose in the lecture on which these notes are based was to report on some highlights of this interplay.

Before engaging directly with the challenge of identifying and exploring the finite images of groups,  
we will step back and consider the following more primitive questions: how should we describe the groups that arise in nature,   and how might we go about the task of understanding a group that is presented to us?
The first question leads us to a 
discussion of group-presentations and the unlimited difficulty of extracting information from such presentations.  
A classical and insightful response to the second  question is that since a group is,   by definition,  
the mathematical object that captures the notion of symmetry,   one ought to explore groups by seeking objects they are symmetries of,   i.e.   one should study {\em actions} on various objects.   But what objects? 
There are many things that are worth trying,  but in the absence of  further information,   the most obvious objects to try first are finite sets,   and that
brings us back to the basic question that frames our study:   to what extent can one understand a group
by studying only its finite images?  

Our focus will be on finitely generated and finitely presented groups.  
We shall discuss how hard it is to identify the finite quotients of a given
group and we shall explore which properties of a group can be detected from its set of finite quotients.  
We will be particularly concerned with questions of {\em profinite rigidity},   i.e.   situations where an infinite group is 
completely determined by its finite quotients.   Groups that arise in low-dimensional geometry and topology
will play a central role in our discussion.   
For example,  
one might ask to what extent the group of symmetries of the paving of the hyperbolic plane shown in figure \ref{fig:334}
is distinguished from all other groups by it finite quotients.   
\begin{figure}[ht]
        \begin{center}
         \includegraphics[width=8.0cm, angle=90]{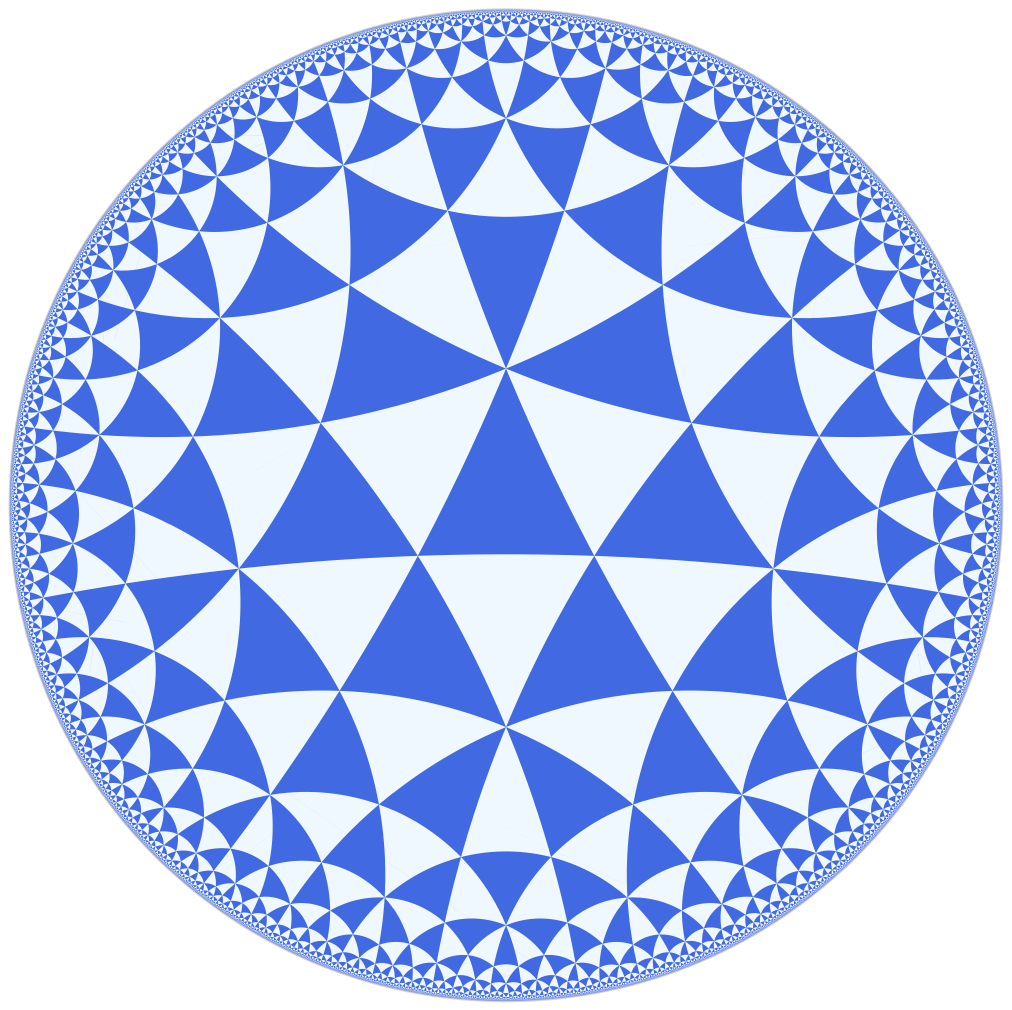}  
                      \end{center}
        \caption{$\Delta(3,  3,  4) = \< A,   B,   C \mid A^3=B^3=C^4=1=ABC\>$}\label{fig:334}
\end{figure} 

The narrative of the lecture on which these notes are based was built around the following sequence of results.   The first
serves to illustrate the difficulty of identifying the finite quotients of a finitely presented group.  

\begin{thmA}[\cite{BW}]\label{t-intro:BW}
There does not exist an algorithm that,   given an arbitrary finite presentation,   can determine whether 
or not the group presented has a non-trivial finite quotient.   
\end{thmA}

Following \cite{BELS},   we shall discuss refinements of this result that describe the (non)existence of algorithms
that can identify quotients in particular families of finite simple groups,   for example the groups ${\rm{PSL}}_n(q)$
with either $n$ fixed and $q$ varying,  or vice versa.

The finite images of a group 
$\G$ form an inverse system indexed by the 
finite-index normal subgroups $N\norm \G$:
 if $N<M$ then $\G/N\onto \G/M$.   The {\em profinite completion} $\wh{\G}$ of $\G$
is the inverse limit of this system; it is a compact totally disconnected
topological group.    There is a natural map $\iota:\G\to\wh{\G}$ with dense image.    
Every homomorphism from $\G$ to a finite
group $Q$ extends,    via $\iota$,    to a map $\wh{\G}\to Q$
 with the same image,  and if $\G$ is finitely generated,  every finite quotient of $\wh{\G}$
arises in this way, so 
  $\G$ and $\wh{\G}$ have the same set of finite quotients.
For finitely generated groups $\G$ and $\Lambda$,   the set of finite images of $\G$ will be the same as the set of finite images of $\Lambda$ if and only if $\wh{\Lambda}$ and $\wh{\G}$ are isomorphic (as abstract or topological groups)
\cite{same-quots, NS};  thus,  for finitely generated groups,  the reader can read the statement $\wh{\Lambda}\cong \wh{\G}$ as ``$\G$ and $\Lambda$ have the same set of finite quotients".  

When discussing profinite rigidity, it is natural to restrict attention to groups $\G$ that are {\em residually finite}, i.e.~every $\gamma\in\G\ssm\{1\}$ has non-trivial image in some finite quotient of $\G$.
We shall consider three versions of profinite rigidity:
{\em absolute} profinite rigidity,   in which comparisons to all finitely generated, residually finite groups are made;
 {\em relative} profinite rigidity,  
whereby one compares groups within a class  -- e.g.   3-manifold groups,   nilpotent groups,    or lattices in Lie groups; 
and {\em Grothendieck} rigidity,   where one is concerned with the existence of pairs of non-isomorphic
groups $P\hookrightarrow \G$
such the inclusion induces an isomorphism of profinite completions.   A central result in this last setting is the following.  

\begin{thmA}[\cite{BG}]\label{t-intro:BG}
There exist residually finite (Gromov) {hyperbolic groups} $H$ 
 and  finitely presented subgroups  $P\hookrightarrow \G:=H\times H$ of infinite index,  
 with $P\not\cong\Gamma$,    such that $u:P\hookrightarrow \G$ induces an
isomorphism $\hat u:\hat P\to \hat\G$.   (Moreover, one can 
\end{thmA} 
\vskip -0.2cm
We shall see in Section \ref{s:groth} that in this theorem one can assume $H<{\rm{SL}}(n,  \mathbb Z)$. 
\smallskip

Easy examples of residually finite groups that are profinitely rigid in the absolute sense include $\Z^d$,   but it is much more
difficult to find {\em full-sized} examples,   i.e.   groups that contain a non-abelian free subgroup and hence do not satisfy
any group law.   This is where   input from low-dimensional geometry and topology comes to the fore.  
The reader will recall that ${{\rm{PSL}}}(2,\R)$  and ${\rm{PSL}}(2,  \C)$ are the groups of orientation preserving isometries
of the hyperbolic plane $\H^2$ and hyperbolic 3-space $\H^3$, respectively.

\begin{thmA} [\cite{BMRS2}] \label{t-intro:BMRS} There are arithmetic lattices  in ${\rm{PSL}}(2,  \C)$
and ${{\rm{PSL}}}(2,\R)$  that are profinitely rigid in the absolute sense. 
\end{thmA}

Explicit examples of such lattices in ${\rm{PSL}}(2,  \C)$
 include the Bianchi group ${\rm{PSL}}(2,  \Z[\omega])$,   where $1+\omega +\omega^2=1$,  
as well as certain cocompact lattices of small covolume.   Examples in ${\rm{PSL}}(2, \R)$
include the triangle group $\Delta(3,  3,  4)$  portrayed in
figure 1.

The remarkable nature of 3-manifold groups and the beauty and utility of hyperbolic geometry and its generalisations
have been major themes in the study of profinite rigidity over recent years.  Another important theme that has emerged is the importance of finiteness properties: there can be a stark difference in the behaviour of  finitely generated and finitely presented groups   in this context.
The following recent result is perhaps
the best illustration of that theme.   A more refined version of this theorem,   which again relies on extensive input
concerning 3-dimensional manifolds,   will be described in the penultimate section of these notes.  

\begin{thmA}[\cite{BRS}]\label{t-intro:BRS} There are finitely presented,   residually finite groups $\G$ that
are profinitely rigid among finitely presented groups but contain infinitely many non-isomorphic finitely generated
subgroups $H\hookrightarrow\G$ such that $\wh{H}\cong\wh{\G}$
\end{thmA}

\noindent{\bf{Acknowledgements:}} I thank my longstanding co-author Alan Reid for  
the many insights that he has shared as we have explored profinite rigidity together.
I acknowledge with gratitude my debt to the late Fritz Grunewald, who enticed me into this field.
I also offer warm thanks to my other co-authors: 
Marston Conder,  David Evans,  Martin Liebeck,  Ben McReynolds,  Dan Segal, Ryan Spitler, and Henry Wilton.
In addition,  I thank Henry for his helpful comments on my first draft of these notes.

I was honoured by the invitation to speak at the 9th European Congress of Mathematics in Sevilla and I 
thank the organisers for making it such a stimulating and enjoyable event.  I am also grateful to
Stanford University and my hosts Tadashi Tokieda and Gregory Lieb for the hospitality 
that enabled me to finish this manuscript.

\section{Presenting Groups and Understanding Them}\label{s:presentation}

In a moment we will decipher the caption of figure 1 by recalling what it means to {\em present} 
a group.  
But first let me provoke the reader with  the assertion that group theory is not really a branch of algebra -- it
 belongs to the whole of mathematics,   just as numbers do.  

\subsection{Groups everywhere} What kind of mathematics do you want to do?  When we start doing formal mathematics,    settling down to lend precision to an intuition gleaned from examples,   we make precise
definitions that encapsulate the essence of the objects that we want to
study.    
As we start to manipulate
and compare our objects,   we have to decide what sort of maps  $X\to Y$ we will allow.  
If we are studying sets,   then any maps will do,   but if we are interested in
a linear problem and are studying vector spaces,   we are likely to restrict to linear or affine maps; if we are doing topology we will probably consider continuous maps; 
if we are modelling flexible geometric objects we might study Lipschitz maps,   {\em{etc. }.}   But 
no matter what the context,   it is 
likely that sooner or later we will want to understand the {\em automorphisms} of our objects $X$ -- i.e.   the set of 
invertible maps (of our chosen type) that preserve the defining features of $X$.   And,   no matter what the context,  
these automorphisms form a {\em group},   $\aut(X)$.    
According to one's nature,   one might then be drawn to the study of groups themselves.

\subsection{How should we capture groups?} 
Maintaining the spirit of doing mathematics in the wild,   let's consider how we might go about describing the symmetries of objects we encounter in nature. 
For an ethane molecule,  we might start by noting that it has six  (rigid, orientation-preserving) symmetries,
but more information is contained in the observation that
we need to combine two   basic operations -- $T$ (twist) and $F$ (flip) -- in order to get everything.  
We also note that doing $T$ three times or $F$ twice amounts to doing
nothing and,   after some more thought,   we might also observe that doing $F$ then $T$ then $F$ has the same effect  as doing $T$ twice.    In summary,   $\< T,   F \mid T^3=1=F^2,  \ FTF=T^2\>.  $
\begin{figure}[ht]
        \begin{center}
                   \includegraphics[width=2.4in]{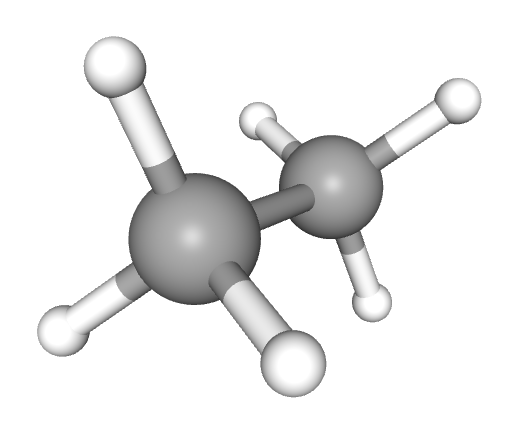}
                      \end{center} 
                       \caption{$\< T,   F \mid T^3=1=F^2,  \ FTF=T^2\>.  $} 
\end{figure}    
We want to formalise the fact that these notes describe the group completely.  
 
\subsection{Presentations of Groups} 

A {\em{presentation}} of a group $G$ is an expression of the form 
\begin{equation*}
G= \;  \langle {a_{1},   a_2,  \ldots}\mid
{r_1=1,  \,   r_2=1,  \dots }\rangle
\end{equation*}  
where the $r_j$ are finite words in the letters $a_i$ and their inverses.   The letters $a_i$ are called
{\em generators} and the $r_j$ are called {\em relators}  (or ``defining relations").  
We will be mostly concerned with the case where the lists $(a_i)$ and $(r_j)$ are finite,   in which
case the presentation is said to be finite and $G$ is said to be {\em finitely presented}.   

If $r_j$ is a concatenation of words $u_jv_j$,   it is sometimes convenient to write
$u_j=v_j^{-1}$ instead of $r_j=1$.   It is also common practice to abbreviate the list of relators by writing
$r_1,  r_2,  \dots$ instead of $r_1=1,  r_2=1,  \dots$.   

When we give a presentation,   we are asserting  three things.   First,   
 there is a choice of elements $a_1,  a_2\dots\in G$
(which need not be distinct)
such that  every element of $G$ can be obtained by repeatedly performing the
operations  $a_i$ and their inverses.     Two assertions are made concerning the
 relators: first,    each  of the equalities $r_i=1$ is true in the group $G$ (where $1$ represents the identity element); 
 second,   {\em{all}} relations among  the $a_i^{\pm 1}$ in $G$ are  consequences of the rules $r_j=1$,   where a
 ``consequence" is a deduction made by repeatedly appealing to the following obvious facts and nothing else
$$  [ r= s\ \ \ {\text{and}}\ \ \ t=u  \ \ \ \implies rt=su],  
 \  \ \ \ [ r= 1\ \ \ \ \implies a^{-1}r a = 1].$$ 
These conditions can be formalised by saying that $G\cong F/R$ where $F$ is the free group
with basis $\{a_1,  a_2,  \dots\}$ and $R\norm F$ is the smallest normal subgroup containing $\{r_1,  r_2,  \dots\}$.  

\smallskip

With these formalities in hand,   we return to the caption in figure \ref{fig:334}.   
The reader can verify 
(with effort) that if one takes $A,  B,  C$ to be suitable rotations about the vertices of any fixed triangle in the tessellation,
then the caption satisfies the three requirements of a presentation.

\subsection{Groups that do not have finite presentations}

A group $G$ is termed {\em finitely generated} if there is a finite collection of elements $a_1,  \dots,  a_n\in G$
such that every element of $G$ is a product of the $a_i$ and their inverses.  
If a group is finitely generated then it is countable,   but not all countable groups are finitely generated.   For example,   it is easy to verify that the rational numbers under addition 
$(\Q,   +)$ cannot be finitely generated.  

It is less obvious how to construct finitely generated groups that cannot be finitely presented,   but 
I want to emphasize that one does not need obscure,   monstrous constructions to find such groups: the absence of
a finite presentation emerges in small,   concrete examples.   To illustrate this,   let
$$ 
{\rm{\bf{X}}}=\begin{pmatrix} 1 & 2\cr 0& 1\cr\end{pmatrix},  \ \ \ \ \ {\rm{\bf{Y}}}=\begin{pmatrix} 1& 0\cr 2& 1\cr\end{pmatrix}
$$
and consider the $4$-by-$4$ integer matrices 
$$
A=\begin{pmatrix} {\rm{\bf{X}}} & 0\cr 0& {\bf{I}}\cr\end{pmatrix},  \ \ \  B=\begin{pmatrix} {\bf{I}}& 0\cr 0& {\rm{\bf{X}}}\cr\end{pmatrix},  \ \ \ 
C =\begin{pmatrix} {\rm{\bf{Y}}}& 0\cr 0&{\rm{\bf{Y}}}\cr\end{pmatrix}\  \ .  
$$
\begin{example}\label{ex:4}
The subgroup of $ {\rm{SL}}(4,  \Z)$ generated by $\{A,  B,  C\}$ has {\em{no finite presentation}}.  
To prove this,   one first uses Klein's ping-pong argument
to prove that the subgroup of ${\rm{SL}}(2,  \Z)$ generated by $X$ and $Y$ is a free group $F$ of rank $2$.   
Next,   one observes that $A,   B,   C$ lie in an obvious copy of $F\times F$.   In fact,   $\<A,  B,  C\>$ is normal in
$F\times F$ with quotient $\Z$,   and such a subgroup can never be finitely presented \cite{BaRo,   grun}.  
This example,   which was studied by Stallings \cite{stall},   is the seed for a rich vein of ideas in geometric
group theory involving higher-dimensional (homological) finiteness properties \cite{BestBrady},   
algebraic fibring \cite{BNS,   kielak},   and residually-free groups \cite{BHMS}.  

\end{example} 

The indirect nature of the proof sketched in the preceding example points to the fact that it
can be {{hard}} to decide if  an explicitly-described group has a  finite  presentation.   The fact that the following question  has defied our understanding for half a century underscores this point.
\begin{question}[Serre 1973 \cite{serre1974}]
Can every finitely generated subgroup of ${\rm{SL}}(3,  \Z)$ be finitely presented?
\end{question}  

The following theorem points to deeper problems and introduces a theme that we will pursue in the next section.  
This result relies on the existence of groups with relatively small presentations where the
word problem is unsolvable,   e.g.   \cite{boris}.   It is contained in the work of Mihailova \cite{mihailova} and Miller \cite{miller}; for a concise proof see \cite[Lemma 2.1]{mb:raag}.   

\begin{theorem} 
Let $F_2$ be a free group of rank $2$.  
There is no algorithm that,   given a set of $20$ elements  $S\subset F_2\times F_2$ can
decide if the subgroup $\<S\>$ has a finite presentation or not.  
\end{theorem}

Even in situations
where one knows that finite presentations of groups exist,   it can be extremely hard to construct them -- for example,   it is  far from clear how to construct an explicit
 finite presentation for any uniform lattice\footnote{A {\em lattice}  in a Lie group $G$ is a discrete
 subgroup $\G<G$ whose co-volume ${\rm{vol}}(G/\G)$, with respect to Haar measure,  is finite.  The lattice is
 {\em{uniform}} if $G/\G$ is compact and otherwise it is non-uniform. } $\G < {\rm{SL}}(3,  \R)$.  
Moving beyond lattices,   if one works only 
with groups of integer matrices given by finite subsets $S<{\rm{SL}}(d,  \Z)$,    
then one can prove that there is no algorithm that allows one to construct finite presentations,   even in 
cases where one knows that such presentations exist \cite{BW:ggd}.

\subsection{Problem: hard to extract information from a presentation}

Notwithstanding the preceding comments on the difficulty of finding presentations,   
when groups arise in nature (e.g.   as discrete groups of automorphisms of a nice space),   there are some general
procedures that one can try to use to construct presentations,   with versions of the Seifert/van Kampen theorem
foremost among them.   These are particularly effective in low dimensions.    There are also various algebraic
settings in which groups come naturally equipped with a finite presentation.   
But it can be ferociously hard to extract information about a group from a presentation.   
Rather than engaging in an abstract discussion of this difficulty,   let's examine three similar-looking small presentations:
\def\a{\alpha}
\def\b{\beta}
\def\c{\gamma}
\def\d{\delta} 

\begin{equation} \label{e:G_i}
\begin{gathered}
{G_2= \langle A,   B \mid BA = A^2B\rangle}
 \\
{G_3 = \langle a,  b,  c \mid ba=a^2b,  \,   cb=b^2c,  \,   ac=c^2a\rangle}\\
{G_4 = \langle \a,  \b,  \c,  \d \mid \b\a=\a^2\b,  \,   \c\b=\b^2\c,  \,   \d\c=\c^2\d,  
\,   \a\d=\d^2\a\rangle}.
\end{gathered}
\end{equation}
The easiest of these groups to understand is $G_2$: for example,   by adding the relation $A=1$,   we see that $G_2$ maps onto
$\Z$,   so in particular it is infinite.   As for $G_3$ and $G_4$,   one of these groups is infinite
while the other is the trivial group -- {how are we to tell which is which?} 
 In my experience,   asking an audience to vote
on which of $G_3$ and $G_4$ is the trivial group will provide an even split of opinion,   and I see no obvious reason why it should be otherwise.   We'll return to these groups shortly.  

\subsection{We need action!} One can quickly get frustrated by the task of trying to understand 
groups such as  $G_3$ and $G_4$ simply by manipulating
 their presentations algebraically.   Instead,   one seeks to understand a group 
  $G=  \langle {a_{1},  \ldots,   a_{n}}\mid
{r_1,  \dots,  r_m}\rangle$ by examining the ways in which it can be realised as a group of symmetries
(automorphisms) of something.   For example,   one might pursue the path of classical representation theory by trying to
understand linear actions of $G$: in finite dimensions,   this amounts to searching for invertible matrices $A_1,  \dots,  A_n$ 
so that substituting each $A_i$ for $a_i$ in the defining relations $r_j$ yields a product that equals the identity matrix.  
Most primitively of all,   we can examine the ways in which $G$ can act as a group of permutations of a finite set.  
Beyond this,   we might look for actions  on spaces that carry  interesting or illuminating geometric,   topological or
analytic features; this is where geometric group theory really begins.   
 
\subsection{A group springing into action: topological models}

Instead of viewing $G=\<a_1,  \dots,  a_n\mid r_1,  \dots,  r_m\>$ as a recipe for expressing $G$ as a quotient
of a free group,   one can regard it as a recipe for building a 2-dimensional 
complex $K$ with fundamental group $G$: the complex has one vertex (0-cell),   has an edge (1-cell) for
each $a_i$,   orientated and labelled $a_i$,   and for each relator $r_j$ there is  2-cell  obtained by attaching
a disc to the loop in the 1-skeleton that spells out the  word $r_j$.    If one does this with 
the standard presentation $\Z^2=\langle a,  b \mid aba^{-1}b^{-1} \rangle$,   the complex $K$ is a torus.  
One can then unwrap $K$ to its universal cover to get the group $\Z^2=\pi_1K$ acting freely and 
cocompactly\footnote{Thanks to Tim Riley for drawing this picture} (as the group of deck transformations).  
\begin{figure}[ht]
        \begin{center}
\psfrag{a}{{$a$}}
\psfrag{b}{{$b$}}
\includegraphics[width=3.7in]{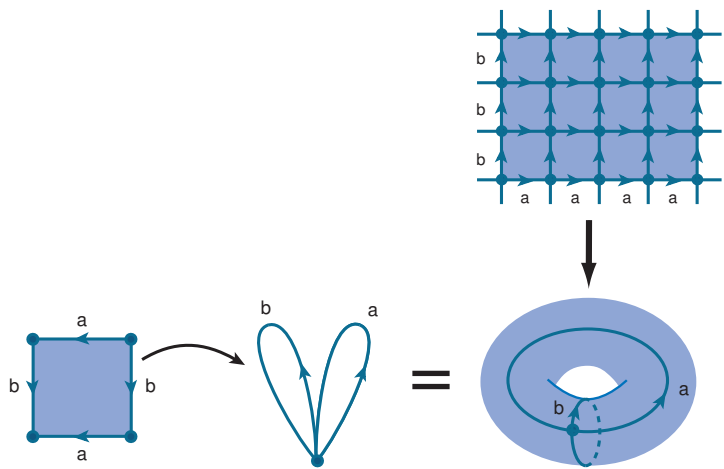}
        \end{center}
\caption{The 2-complex creating an action of $\Z^2=\langle a,  b \mid ab=ba \rangle$} 
\end{figure}  
If we start with an  arbitrary finite presentation,   the 2-complex $K$ will not be a manifold in general,   
but one can remedy this by embedding  it in  $\R^5$ then taking the boundary of a regular neighbourhood.   This boundary can   be
smoothed and with a little effort one can argue that it has the same fundamental group as the 2-complex
we started with,   thus proving:

\begin{proposition} Every finitely presented group is the fundamental group of a smooth,   compact
4-dimensional manifold.    
\end{proposition}

Classes of groups that are of special interest emerge when one asks how one might improve
on this construction.   
The most obvious question,   perhaps,   is whether one can reduce to three dimensions;  the
answer is no,   which provides us with the first hint that fundamental groups of compact 3-dimensional manifolds
(``{\em{3-manifold groups}}" for short) have special features,   a theme to which we will return.

But first let me remark that while these general
 constructions projecting the study of group presentations into topology are both useful and appealing,    
no hard work has been done,   so one should expect them to 
preserve the hardness of most problems.   
Thus,   for example,   the task we set ourselves of deciding which of $G_3$ and
$G_4$ was the trivial group is faithfully translated into the equally hard problem of deciding which 
of the 2-complexes built from these presentations is contractible,   i.e.   can be continuously deformed
within itself to a point.   (It requires some theory to see that these problems are equivalent.  )  
   
\subsection{Decision Problems (Max Dehn 1911)}

Combinatorial group theory is the study of groups given by generators and relators.   Its origins are
intertwined with the development of low-dimensional topology at the beginning of the twentieth century,   particularly in the work of Max Dehn who,  in his famous paper \cite{dehn},  was the first to articulate clearly the algorithmic problems involved
in extracting information from finite presentations $G = \langle   {a_{1},  \ldots,   a_{n}}\mid
 {r_1,  \dots,  r_m}\rangle $,  singling out three problems as the most fundamental:  the {\em{Word Problem}} (deciding which words in the generators equal $1\in G$),  the {\em{Conjugacy Problem}} (deciding which words represent conjugate elements in $G$),  and the {\em{Isomorphism Problem}} (which includes the Triviality Problem -- deciding if the
given group is trivial or not).
One senses a  certain level of frustration when he writes:
{\em{``One is led to such problems by necessity when working in geometry and topology."}}
His main interest at the time was the classification of knots and related questions in 3-manifold topology.
  
\subsection{Unsolvable Problems}\label{ss:parlay}
 In the 1950s,   
P.S. ~Novikov \cite{novikov} and W.W. ~Boone \cite{boone} (followed by Adian \cite{adian},  Rabin  \cite{rabin} and others)
proved that the sought-after algorithms do not exist in general.   
  
  \begin{theorem}[Noviko-Boone] There exist finitely presented groups $G=\<a_1,  \dots,  a_n\mid r_1,  \dots,   r_m\>$
  for which there is no algorithm that,   given an arbitrary word $w$ in the generators,   can
  decide whether or not $w=1$ in $G$.   
  \end{theorem} 

Starting with a finitely presented group $G$ that has an unsolvable word problem,   it is not difficult to construct
groups that have solvable word  problem but unsolvable conjugacy problem,   nor is it too difficult to construct
recursive sequences of finitely presented groups so that one cannot tell which are trivial.  
For the triviality problem, one uses  the theory of HNN extensions and amalgamated free products to make a 
sequence of groups, indexed by words in the generators of $G$, so that $w=1$ in $G \iff G_w=\{1\}$.

\subsection{3-manifold groups are well behaved}

\medskip
We have already seen that 3-manifold groups are special,   and this is particularly borne out in the context
of decision problems,   although it took the whole of the twentieth century to prove that the algorithms
that Dehn sought in this context at the beginning of the century do exist.    

\begin{theorem} If $\G$ is the fundamental group of
 a compact 3-manifold,   then there are algorithms to solve the word and conjugacy problems in $\G$.   
 Moreover,   there is an algorithm that,   given two finite presentations of 3-manifold groups will
 determine whether or not the groups presented are isomorphic.  
\end{theorem}

This theorem is a consequence of 
Perelman's resolution of Thurston's Geometrisation Conjecture,   but to see this one requires
the work of many authors; we refer the reader to
\cite{AFW} for a thorough discussion.

\subsection{Returning to the groups $G_2,  \ G_3,  \ G_4$}
 
We commented earlier that the general construction illustrated in figure 3 does not make 
the task of deciding which of $G_3$ and $G_4$ is $\{1\}$ any easier.    Since topology has not helped
us,   we might fall back on representation theory,   or the search for finite images.   But these too fail us:
we can gain further insight into  
${G_2= \langle A,   B \mid BA = A^2B\rangle}$ by noting that the defining relation is satisfied by the
matrices  
$${ {
A=\begin{pmatrix} 1& 1\cr 0& 1\cr\end{pmatrix}\ \ \ \hbox{ and } \ \ \ \ B=\begin{pmatrix} 2& 0\cr 0& 1\cr\end{pmatrix}}},  
$$
yielding a representation that one can prove is faithful.   One can also construct infinitely many finite
quotients,   for example the presentation that we gave of the symmetries of the ethane molecule (figure 2)
provides an action on a finite object.  
But a search for matrices of any size over any
field satisfying the defining relations of $G_3$ and $G_4$ will prove fruitless -- neither group admits 
any non-trivial,   finite dimensional representation over any field.   Likewise,   a search for non-trivial finite
images of these groups will prove fruitless.  

In fact,    $G_3$ is the trivial group (this is a challenging exercise) while $G_4$ is
an infinite group that has no non-trivial finite quotients.   Since $G_4$ has no non-trivial
matrix representations and no non-trivial actions on finite sets,   where are we to look
for an action that will unravel the structure of the group? 
The answer is that we get the
group to act on a {\em{tree}}:
the group decomposes as a non-trivial amalgamated free product and hence  (according to Bass-Serre
theory \cite{serre-trees})
acts
with compact quotient on an infinite tree.   This is an important insight of
 Higman \cite{higman},   reinterpreted by Serre \cite{serre-trees}.   
 
 \subsection{Really useful actions: Trees,  hyperbolicity,  and cube complexes} \label{s:cubes}
 We have just had our first encounter with actions on trees, which is the base example for
 two classes of actions that play a central role in what follows; both originate in the seminal work of Gromov \cite{gromov}.
 First,  there is Gromov's theory of  hyperbolic groups and spaces.
Whereas in a tree all triangles degenerate to tripods, hyperbolicity for geodesic metric spaces is characterised by the
property that all triangles are uniformly thin: there is a constant $\delta>0$ such that each side of each geodesic triangle is
contained in the union of the other two sides. A finitely generated group is {\em{(Gromov)  hyperbolic}} if it
acts properly and cocompactly by isometries on such a  space. It will be helpful for the reader to have a nodding acquaintance with this
theory (which can be gleaned from \cite{gromov} or \cite{BH}),  but no details will be needed to understand  these notes.  

From a different perspective,  one can regard trees as 1-dimensional {\rm{CAT}}$(0)$ cube complexes.  
The rich and extensive theory of groups acting on {\rm{CAT}}$(0)$ cube complexes
has played a central role in geometric group theory and low-dimensional topology
over the last three decades.   The reader unfamiliar with the subject can consult \cite{BH, davis, schwer} for the basic theory
and \cite{sageev} for a concise introduction.   Briefly,   a {\em{non-positively curved cube complex}}
is a space,  built by assembling Euclidean cubes,   glued together by local-isometries between faces,   
so that the resulting path metric on the space is non-positively curved in the sense of A.D. ~Alexandrov,   i.e.  
in the neighbourhood of each point,   geodesic triangles are no fatter than triangles  in the Euclidean plane with the same edge lengths -- see \cite{BH}.   This is equivalent to a purely combinatorial condition on links: there are no ``empty simplicies" in the 
link of any vertex of the complex; in dimension 2, this means that there are no circuits of length less than 4 in the link graphs. 
If the complex is locally finite-dimensional,   all triangles in the universal cover
(not just small ones) will satisfy Alexandrov's comparison condition -- in other words this universal cover will be a {\rm{CAT}}$(0)$ space  \cite{BH}.

There is a rich theory of groups that act by isometries on {\rm{CAT}}$(0)$ spaces and
a  much 
richer theory of groups that act properly and compactly by isometries on {\rm{CAT}}$(0)$  cube complexes, particularly when the
complex is {\em special} in the sense of Haglund and Wise \cite{HW}. This last condition (which concerns the way in which 
cube-bisecting hyperplanes intersect) defines the class of {\em special} groups (sometimes called cocompact special, for clarity). 
Although it is defined in terms of the geometry of hyperplanes, specialness has a global characterisation that is particularly important
for us: a cube complex is special if and only if there is a locally-isometric embedding of it into the Salvetti complex of some right-angled Artin group \cite{HW}.  For us,  the important consequence of this is that any (virtually) special group
can be embedded in ${\rm{SL}}(n,\Z)$ for some $n$.
\begin{figure} 
                   \includegraphics[width=6cm, angle=90]{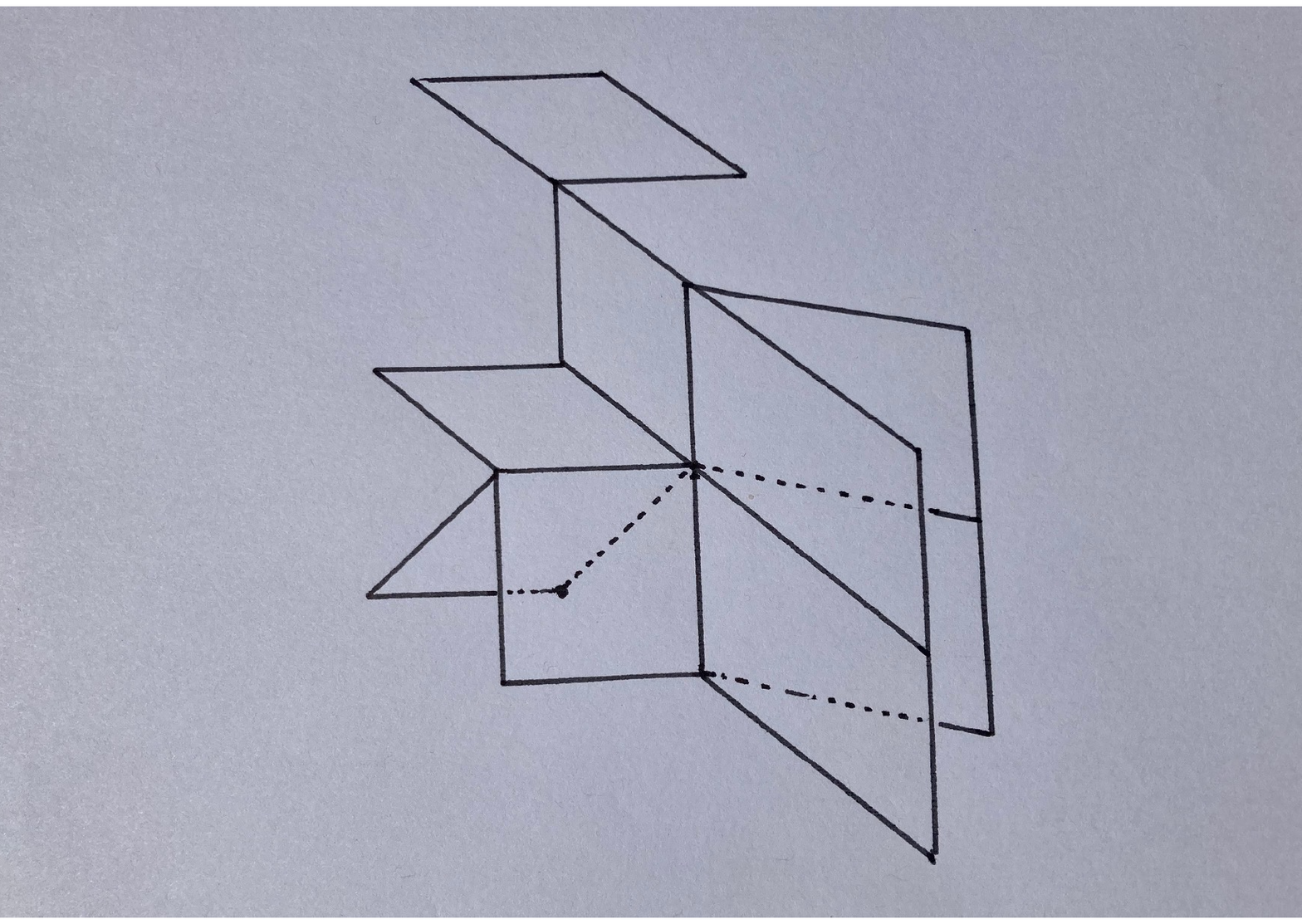}
\caption{A 2-dimensional CAT$(0)$ cube complex} 
 \end{figure}   
 
The powerful theory of (virtually) special groups developed by Wise and his collaborators has had a profound influence
on many aspects of geometric group theory, including the study of profinite rigidity. 
If  the reader explores the original papers on which these notes are based, they will quickly see how useful this theory is. 
A fundamental reason for this is the tight connection between specialness, the behaviour of subgroups in the profinite topology, and
hierarchical decompositions of groups \cite{wise:qc}.

Much of the impact of the theory of specialness in the context of profinite rigidity flows through Agol's Theorem and its consequences \cite{agol}. 
Agol proved that if a hyperbolic group acts properly and cocompactly by isometries on any {\rm{CAT}}$(0)$ cube complex, then it
has a subgroup of finite index that is special. Agol used this theorem to prove Thurston's virtual fibring conjecture for 
hyperbolic 3-manifolds. This fibring theorem, as well as various other consequences of the work of Agol and Wise, are crucial ingredients in
the main results of \cite{BMRS1} and \cite{BRS}.

\section{Hard to see if finite  images exist}

We want to know to what extent  
finitely presented groups can be understood by examining only their finite quotients.  
In this context,   we are inevitably led to ask how hard it is to discern what the finite
quotients are.   There are naive processes that one can run to list all finite quotients of a 
finitely presented group, but if such a process has not found a non-trivial finite quotient
after a certain time,    how is one to know whether it is worth continuing the search?  One cannot know.  

\begin{theorem}[Bridson, Wilton \cite{BW}]\label{t:BW}
There does not exist an algorithm that,   given an arbitrary finite presentation,   can determine whether 
or not the group presented has a non-trivial finite quotient.   
\end{theorem}

A non-trivial finitely generated group of matrices (i.e.   a linear group) always has non-trivial finite quotients,  and finite
groups are linear over any field,  
so Theorem \ref{t:BW} implies that one also cannot decide whether a finitely presented
group has a non-trivial finite-dimensional,    linear representation (over any field or over a fixed field).  

\subsection{Sketch of Proof}

The proof relies on:
\begin{itemize}[leftmargin=*]
\item Logic: Slobodskoi's proof \cite{slobo} of the undecidability of the universal theory of finite groups.
\item Technology related to the geometry of 
subgroup separability in low-dimensions,  particularly  in the context of graphs of free and 
finite groups,  and their covering spaces,   as well as
 {\rm{CAT}}$(0)$ cube complexes (ideas of Stallings \cite{stall-graphs} and Wise \cite{wise:omni, wise-invent}).
\end{itemize}

The {\em{universal theory of finite groups}} is the set of all one-quantifier first-order sentences that are true in all finite
groups.   To see why it is relevant here,   consider the following sentence  $\Psi$ (remembering that
$\vee$ means ``or").   
$$   
\forall a,b,c,d: (ba^2b^{-1}\neq a^3)\, \vee\,(dc^2d^{-1}\neq c^3)\,\vee\,([a,b]\neq d)\,\vee\,([c,d]\neq b)\,\vee\,(a=b=c=d=1).
$$
The reader can convince themself that
this sentence is true in a specific group $Q$ if and only if there is no non-trivial homomorphism
$B\to Q$,   where $B$ is the group
\begin{equation}\label{e:B}
B = \langle a,b,c,d\mid ba^2b^{-1}a^{-3},\, dc^2d^{-1}c^{-3},\, [a,b]d^{-1},\, [c,d]b^{-1}\rangle.
\end{equation}
The reader will then see how to write a sentence $\Phi_{AR}$ based
on any finite group-presentation $\<A\mid R\>$ so that $\Psi_{AR}$ is true in a group $Q$ if any only
if there is no non-trivial homomorphism $G\to Q$,   where $G=\<A\mid R\>$.   

In our example,   the group $B$ (which is infinite) has no finite quotients \cite{BG},   so the sentence
$\Psi$ is true in every finite group,   i.e.   it  belongs to the   universal theory of finite groups.  In 
proving that this theory is undecidable, 
Slobodskoi \cite{slobo} constructed a finitely presented group $G$  in which there is no algorithm to determine 
which elements of the group survive in any finite quotient.   One can phrase this as the unsolvability of a profinite
version of the word problem in $G$.   
 With such a group $G$ in hand,   one might  hope
 to prove Theorem \ref{t:BW} by
using amalgamated free products and HNN extensions in the manner alluded to in Section \ref{ss:parlay},
building   a sequence of finitely presented groups $G_w$,   indexed by words in the generators of $G$,  
so that $G_w$ has a non-trivial finite quotient if and only if the image of $w$ is non-trivial in some finite
quotient of $G$: in a standard notation,  $w=1{\text{ in }}\widehat{G}\iff \widehat{G}_w=1$.  

Ultimately,   this outline of strategy does succeed (see \cite[Theorem C]{BW}),   but the
details of the implication $\implies$ are much harder than in the classical situation (\ref{ss:parlay}.)   This
is where the technology of Stallings  and Wise  is used.  
For reasons that quickly become clear when one tries to pursue the natural strategy,   a key point is that one
needs to be able to control the orders of the generators of $G$ (and related groups) in finite quotients: this
requires $G$ to be crafted carefully,   and it relies on refinements of  Wise's work on {\em omnipotence} \cite{wise:omni},  
which began with his proof that,   given elements $a,  b$ of a non-abelian free group that are independent in homology,   one 
can find a finite quotient in which the orders of $a$ and $b$ are in any given ratio.   Beyond this first stage of control,   one needs further technology concerning malnormal subgroups of virtually free groups,   
fibre products,   and theorems concerning malnormal amalgams and residual finiteness.

\subsection{Refinement: Finite Simple Images} 

It follows easily from Theorem \ref{t:BW} that there is no algorithm that determines the finite simple images of a finitely presented group.   Famously,   every non-abelian finite simple group is either an alternating group,   a group 
from one of 16 families of Lie type,   for example ${\rm{PSL}}_d(q)$,   or one of the 26 sporadic groups.   Thus the question arises: for which collections of finite simple groups is there an algorithm that determines the members of the collection that are images of a  given finitely presented group,   and for which families do such algorithms not exist?  This question
is essentially settled in \cite{BELS} where it is proved that if a collection of finite simple groups contains infinitely many alternating groups,   or contains classical groups of unbounded dimensions,   then there is no algorithm,   whereas  for families of simple groups of  Lie type with bounded rank,   the desired algorithms do exist.   The following special case of
this result provides a concrete illustration of this dichotomy.    

\begin{theorem}[Bridson,  Evans,  Liebeck,  Segal \cite{BELS}]$\ $
\begin{enumerate}[leftmargin=*]
\item There is an algorithm that,   given a finitely presented group $\G$ will
 determine whether or not $\G$ has infinitely many quotients ${\rm{PSL}}_n(q)$ with $n$ fixed and $q$ varying,   and
 will list these quotients if there are only finitely many.  
\item There does not exist an algorithm that,   given a finitely presented group $\G$,   will
 determine whether or not $\G$ has infinitely many quotients ${\rm{PSL}}_n(q)$  with $q$ fixed and $n$ varying,   nor
 does there exist an algorithm that can determine whether $\G$ has any quotient of this form.  
 \end{enumerate}
\end{theorem}
 The proof of (1) is rooted in representation theory and the model theory of  finite fields; the algorithm is not practical.
    
\section{Capturing groups via finite actions: Residual Finiteness}

We return to our main theme:  
{\em{To what extent is a finitely generated (or finitely presented) group determined by its finite
quotients (equivalently,  its profinite completion)?}} 

The information contained in the kernel of the natural map $\G\to\wh{\G}$ is lost when we restrict our
attention to $\wh{\G}$,  so it is natural to focus on groups where this kernel is trivial,  i.e.  the   residually
finite groups.

\begin{definition} A group
$\G$ is {\em residually finite} if for every $ \gamma\in\G\smallsetminus\{1\}$
there is a homomorphism  $\pi:\G\to Q$ to a finite group with $\pi(\gamma)\neq 1$. 
\end{definition}

\subsection{Which groups are residually finite?}\label{ss:rf}
There are many results in the literature establishing that groups are residually finite; the
following selection is slanted towards the theme of these notes.  I particularly want to 
emphasize the remarkable nature of items (4) and (5): they provide combinatorial
criteria for verifying that a group is linear, and they have vastly extended the available reservoir of
residually finite groups.  The wealth of examples that these criteria provide,  made available largely through the work of
Dani Wise \cite{wise:qc},  underpins many advances in the field.
 
\begin{enumerate}[leftmargin=*]
\item
If a finitely generated group $\G$ admits a faithful linear representation over any field $\Gamma
\hookrightarrow{\rm{GL}}(n  K)$,   then $\G$ is residually finite \cite{malcev}.
\item
If $\G$ is the fundamental group of a compact 3-manifold,   then $\G$ is residually finite \cite{hempel:rf}. 
\item
Mapping class groups of compact surfaces and (outer) automorphism groups of finitely
generated free groups are residually finite \cite{gross}.
\item Hyperbolic groups that act properly and cocompactly on CAT$(0)$ complexes are residually finite.
\item $C'(1/6)$ small-cancellation groups are residually finite.
\item Random finitely presented groups,  in a suitable model,  are residually finite.
\end{enumerate}

(4) is a consequence of Agol's Theorem, as discussed in Section \ref{s:cubes}.
The  $C'(1/6)$ small-cancellation condition in (5) bounds the size of overlaps between subwords of relators in group-presentations.
Groups that have finite presentations satisfying this condition are hyperbolic, and Wise \cite{wise-sc} proved
that they act on CAT$(0)$ cube complexes,  so (4) applies. This way of phrasing things, though,
obscures the fact that Wise's earlier work
established  residual finiteness in many cases,  opening up the great reservoir of examples that (4) and (5)
provide \cite{wise-invent, wise:rf, wise:qc}.   
 A popular model for a random group  is Gromov's density model \cite{gromov2}, whereby
one fixes the number of generators for a presentation and chooses the relators uniformly at random from 
cyclically-reduced words of a given
length, with the number of words determined by a density parameter; the focus is on what happens as the length gets large. At small densities,  with overwhelming 
probability these groups satisfy the  $C'(1/6)$ small-cancellation condition \cite{Oll}.
 
The following challenge is one of the guiding problems of geometric group theory.

\begin{question} Are all Gromov hyperbolic groups  residually finite?  \end{question}

Residually finite groups harbour less pathology than arbitrary groups.  For example,
it is easy to see that a finitely presented,  residually finite group $\G$ has a solvable word problem,
although there do  exist finitely generated,   residually finite groups that have unsolvable word problem \cite{meskin}.  
Note the recurrence of an important theme here:  finitely generated groups can be much
wilder than finitely presented ones. 

We mention one other property of residually finite groups because it plays a surprisingly useful 
role in the results we shall look at later.  

\begin{lemma}\label{l:hopf}
If a finitely generated group $\G$ is residually finite,   then it has the Hopf property,   i.e.   every epimorphism
$\phi:\G\to \G$ is  an isomorphism.  
\end{lemma}

\begin{proof} If there were a non-trivial
element $\gamma\in\G$  in the kernel of an
epimorphism $\phi:\G\to\G$ and  $\pi: \G\to Q$ was a homomorphism to a finite group with $\pi(\gamma)\neq 1$,
then the infinitely many maps $\pi\circ\phi^n:\G\to Q$ would all be distinct,   which is impossible
because any  homomorphism $\G\to Q$ is determined by the image of a finite generating set.  
\end{proof}

\section{Profinite properties of groups}
 
 We want to understand the extent to which properties of residually finite groups are encoded in their profinite completions.

\begin{definition}
A property $\mathfrak{P}$ of abstract groups is a {\em profinite property} (or a profinite invariant) if,
for finitely generated,  residually finite groups $\G_1$ and $\G_2$,  whenever $\G_1$ has  $\mathfrak{P}$
and $\wh{\G}_1\cong\wh{\G}_2$,  then $\G_2$ has $\mathfrak{P}$.
\end{definition}

\subsection{Some basic facts  about profinite completion}
In the introduction,  we defined the profinite completion $\wh{\G}$ of an abstract group $\G$ to be
the inverse limit of the system of all finite quotients $\G/N$ of $\G$.  Equivalently,  one can define $\wh{\G}$ 
to be the closure
of the image of $\G$ in the direct product of the finite quotients $\G/N$ (which is compact in the product topology):
$$
\wh{\G} = \{ (x_N) \mid x_NM = x_M\ \forall M>N\}  \subset \prod_N \G/N.
$$
Profinite completion
$\G\mapsto\wh{\G}$ is a functor from the
category of abstract groups to the category of profinite 
groups:
a homomorphism of abstract groups $u:\G_1\to \G_2$ extends uniquely to a continuous homomorphism of
profinite groups $\wh{u}:\wh{\G}_1\to \wh{\G}_2$.   

By construction,  $\wh{\G}$ is residually finite.  $\G$ is residually finite if and only if the natural map $\iota: \G\to\wh{\G}$ is
injective.  Henceforth,  
we shall concentrate on residually finite groups,  in which setting we identify $\G$ with its image under $\iota$.
If $\G$ is finitely generated,  there is a 1-1 correspondence between the finite-index subgroups of $\G$
and the finite-index subgroups\footnote{The
simplicity of this statement relies on the Nikolov-Segal Theorem \cite{NS}.} of $\wh{\G}$: this associates to $H<\G$ its (topological) closure $\overline{H}<\G$, 
and it associates to $\mathcal{H}<\wh{\G}$ the intersection $\mathcal{H}\cap\G$. 

\begin{remark}\label{blob}
An important point to remember is that,  for finitely generated groups $\G_1$ and $\G_2$,  the existence of
an isomorphism $\wh{\G}_1\cong\wh{\G}_2$ does not imply the existence of any non-trivial homomorphism 
from $\G_1$ to $\G_2$.  On the other hand,  the correspondence between finite-index subgroups of $\G_i$ and
$\wh{\G}_i$ does induce,  via $\wh{\G}_1\cong\wh{\G}_2$,  an isomorphism between the lattices of finite-index subgroups
in $\G_1$ and $\G_2$; this isomorphism preserves the index of subgroups,  normality,  and conjugacy.   
It follows that any property that involves counting subgroups of a given index
will be a profinite invariant,  for example subgroup growth \cite{LubSeg}.
Moreover,  if $H_1<\G_1$ corresponds to $H_2<\G_2$,  then $\wh{H}_1\cong\wh{H}_2$ (an observation that can be extremely
useful).
\end{remark}

\subsection{Profinite properties: examples and counterexamples}

The following list contains some properties that are easily seen to be profinite invariants but it also points to many
properties that are  {\em not} profinite invariants.   Rather than being disappointed by this  I hope that
the reader will reflect that the difficulty of finding non-trivial profinite invariants  
makes the existence of profinite rigidity theorems all the more remarkable. 

\begin{enumerate}[leftmargin=*]
\item Being {\bf abelian} is a profinite property. {\em Proof:}
If $\G_1$ is abelian,  then so are all of its finite quotients.  If $\G_2$ is not abelian,  then for
some $a,b\in \G_2$,  the commutator $c=[a,b]$ is non-trivial.  If $\G_2$ is residually finite,  then $p(c)=[p(a), p(b)]\neq 1$
in some finite quotient $p: \G_2\to Q$.  As this quotient is not abelian,  it is not a quotient of $\G_1$,  and therefore 
$\wh{\G}_1\not\cong\wh{\G}_2$.

\item By considering which abelian $p$-groups $\G$ maps onto (varying the prime $p$),  one can show that 
the isomorphism type of the {\bf{abelianisation}} of a finitely generated group is a profinite invariant.

\item A slight variation on the argument in (1) shows that satisfying any {\bf{group law}} is a profinite property,  for example
being nilpotent of class $c$,  or solvable of derived length $d$. It is considerably harder to prove that being
{\bf{polycyclic}} is a profinite invariant, but Gabbagh and Wilson proved that this is true \cite{polyC}. 

\item The $\ell_2$-Betti numbers of groups are defined as analytic invariants, 
 but L\"{u}ck's approximation theorem \cite{luck:gafa} 
allows one to compute the {\bf{first $\ell_2$-Betti number}} of a finitely presented,  residually finite group as a limit of 
ordinary Betti numbers:  $b_1^{(2)}(\G) := \lim b_1(H_n)/[ \G\colon H_n]$,  where $(H_n)$ is any sequence of finite-index
normal subgroups  
intersecting in the identity.   It follows from the facts about subgroup
lattices in Remark \ref{blob} that the first $\ell_2$-Betti number is a profinite invariant among finitely presented,
residually finite groups (\cite[Corollary 3.3]{BCR}).  But Kammeyer and Sauer show
that the higher-dimensional $\ell_2$-Betti numbers are
not profinite invariants  \cite{KaS}.

\item Lackenby \cite{lack} proved that having a subgroup of finite index that maps onto a non-abelian free group
({\bf{largeness}}) is a profinite invariant among finitely presented groups.

\item Aka \cite{aka1} found lattices in semisimple Lie groups that have the same profinite completion even though
the real-rank of the Lie groups is different.  He also  proved that Kazhdan's {\bf{property (T)}} is not a
profinite invariant \cite{aka2}.  Work of Kassabov and Nikolov \cite{KN} shows that property $(\tau)$ is not a profinite invariant.

\item Property (T) can be phrased as a fixed point property for isometric actions on Hilbert spaces.
{\bf{Fixed-point properties}} for actions on other CAT$(0)$ spaces have also been studied. 
Cheetham-West,  Lubotzky, Reid and Spitler \cite{tam} showed that Serre's property FA (the fixed-point property
for actions on tress) is not a profinite
invariant,  and in \cite{mrb:agt} this was extended to fixed-point properties for actions on higher-dimensional CAT$(0)$
spaces.

\item Kionke and Schesler \cite{KS} proved that {\bf{amenability}} is not a profinite invariant. 

\item Being finitely presented is not a profinite invariant of finitely generated groups (see Section \ref{s:groth}),
and likewise the higher {\bf{finiteness properties}} of groups \cite{lub-FPn}.

\item Being {\bf{hyperbolic}} (in the sense of Gromov) is not a profinite invariant among finitely generated groups, 
but it is unknown whether it is preserved under profinite equivalence of finitely presented groups.

\item It is unknown whether having a finite dimensional {\bf{linear}} representation is a profinite property.

\item The theorem stated below can be used to see that the property 
of having a {\bf{subgroup}} of a given sort is almost never a profinite invariant: for example,  having a non-trivial 
finite subgroup,  or a subgroup isomorphic to $\Z^3$,  or a 2-generator subgroup that is neither abelian nor free.
This precludes the profinite invariance of various properties that are inherited by subgroups, e.g.~being orderable,   or locally indicable,  or having cohomological dimension less than $d>2$.
\end{enumerate}

\begin{theorem}[\cite{mb:FbyF}]  \label{t:FbyF}
Given a finitely generated,  recursively presented  group $\G$
that is residually finite,  one can construct a finitely generated,  residually finite
free-by-free group $M_\Gamma = F_\infty\rtimes F_4$ and an embedding
$M_\Gamma  \hookrightarrow (F_4\ast \G)\times F_4$
that induces an isomorphism of profinite completions.
\end{theorem}

This theorem is proved by combining ideas described in Section \ref{s:groth} with the universal
embedding theorem in \cite{mb:baum},  which elaborates on ideas of Higman \cite{hig} and
Baumslag,  Dyer and Heller \cite{BDH}.

\section{Profinite Rigidity}

When studying profinite rigidty, one is interested in theorems whose conclusion 
is $${\wh{\La}}\cong{\wh{\G}}\implies \La\cong\G.$$

\subsection{Three Forms of Rigidity}

The following definition of what it means for a group to be profinitely rigid has become standard, although it is
still common to add ``in the absolute sense" for clarity (as we shall do several times in this article).
 
 \begin{definition}[{\bf{Absolute}} Profinite Rigidity]\label{d:absolute}
  A finitely generated,   residually finite group $\G$ is
  {\em profinitely rigid}
 if,   whenever $\La$ is finitely generated and   residually finite,  
$
\wh{\La}\cong\wh{\G}\implies \La\cong\G. 
$
\end{definition}
In the previous section we saw that when $\G$ satisfies a group law,  one only needs to consider groups $\Lambda$
satisfying the same group law.  In such cases,  (\ref{d:absolute}) reduces to
an instance of the following problem. 

\begin{definition}[{\bf{Relative}} Profinite Rigidity in a class $\mathcal{G}$] 
If a finitely  generated,   residually finite group $\G$ belongs to a class of groups  $\mathcal{G}$, then
one says that $\G$ is {\em profinitely rigid among groups in $\mathcal{G}$}
if whenever $\La\in\mathcal{G}$  is finitely generated and   residually finite,     $
\wh{\La}\cong\wh{\G}\implies \La\cong\G. 
$ 
\end{definition} 

\begin{remark}[The absolute case is much harder]\label{lemon-squeezy}
The literature contains many more theorems about relative profinite rigidity than it does about absolute profinite
rigidity,  for the obvious reason that even if one starts with a group $\G$ that has rich structure,
when one finds a group $\Lambda$ ``lying in the gutter" and knows only that it is finitely generated, residually finite and
has the same finite quotients as $\G$, how is one to know that $\Lambda$ shares any of the fine features of $\G$? For example, 
if $\G$ is a group of matrices,  why should $\Lambda$ have any interesting linear representations? If $\G$
is a 3-manifold group,  how would one begin to argue that $\Lambda$ is a 3-manifold group?

In contrast, if one is trying to argue that a group is profinitely rigid in a class $\mathcal{G}$,
then one can study special features of the groups in $\mathcal{G}$ and ask if these are visible in $\wh{\Lambda}$
when  $\Lambda\in\mathcal{G}$.  For example, 
we shall see in the Section \ref{s:3M} that many 
features of compact 3-manifolds $M$ are invariants of $\wh{\pi_1M}$.
\end{remark} 

The third form of profinite rigidity that is widely studied is {\bf{Grothendieck}} rigidity, where  one is interested in whether
it is possible to obtain isomorphisms $\wh{u}:\wh{\Lambda}\to\wh{\G}$ induced by homomorphisms 
of discrete groups $u:\Lambda\to\Gamma$. This  is the subject of Section \ref{s:groth}.

\subsection{Early work,  sobering examples, and a guiding question}\label{ss:sobering}

\begin{enumerate}[leftmargin=*]
\item We have seen that finitely generated abelian groups are profinitely rigid.
When one moves from abelian to {\bf{nilpotent}} groups,  one finds that some are profinitely rigid (e.g. free nilpotent groups) while others
are not \cite{remes}.  But in all cases, the ambiguity is limited: only finitely many finitely generated nilpotent
groups can have the same profinite completion \cite{pickel}.
Following \cite{GZ},
an appealing way to express this finite ambiguity is to use the language of {\em profinite genus} (within a class of groups). 
Thus Pickel's Theorem  is that the profinite genus of any finitely generated nilpotent group is finite.
Grunewald, Pickel and Segal \cite{GPS} proved that  {\bf{virtually polycyclic}} groups also have {\bf{finite genus}}.

\item  In 1974 Baumslag \cite{Baum74}  discovered the following striking example which
underscores how quickly one can lose profinite rigidity when one moves away
from the abelian setting.
Let {\bf{$G_1=(\Z/{25})\rtimes_\alpha\Z$ and $G_2=(\Z/{25})\rtimes_{\beta}\Z$}}, where $\alpha$ is the automorphism
$x\mapsto x^6$ and $\beta=\alpha^2$ is the automorphism $x\mapsto x^{11}$.  The reader can verify directly that
$G_1\not\cong G_2$ but   $G_1$ and $G_2$ have the same finite quotients.

\item Baumslag drew inspiration from an important example of Serre  \cite{serre1964}. In 1964 
Serre constructed a smooth projective variety $V$ defined over a number field $K$ such that for the varieties $V_{\phi_1}$
and $V_{\phi_2}$ obtained by extension of scalars along  different embeddings $\phi_1,   \phi_2 : K\hookrightarrow\mathbb{C}$,  the  fundamental groups $\G_i = \pi_1 V_{\phi_i} \cong \Z^{p-1}\rtimes  (\Z/p)$  are not isomorphic,
but $\wh{\G}_1 \cong \wh{\G}_2$ because  each $\wh{\G}_ i$ is the algebraic (\'{e}tale) fundamental group of the underlying scheme.  
Pairs of {\bf{Galois conjugate}} varieties exhibiting this behaviour can also be found among complex surfaces \cite{BCG}. 
Galois conjugation will play an important role in Section \ref{s:Galois}.

\item  The examples of Stebe \cite{stebe} are based on a different phenomenon. 
He constructed pairs of  matrices in ${\rm{SL}}(2,\Z)$ that are not conjugate in $
{\rm{GL}}(2,\Z)$ but are conjugate when one projects them to ${\rm{GL}}(2,\Z/d)$ for every integer $d>2$.  One such pair is
$$
\phi(1)=\begin{pmatrix} 188& 275\cr 121& 177\cr\end{pmatrix}\ \ \ \hbox{ and } \ \ \ \ \phi(2)=\begin{pmatrix} 188& 11\cr 3025& 177\cr\end{pmatrix}.
$$
Writing {\bf{$\G_i = \Z^2\rtimes_{\phi(i)}\Z$}},  we have  $\G_1\not\cong\G_2$ {but} $\widehat{\G}_1 \cong\widehat{\G}_2.
$
Every element of ${\rm{SL}}(2,\Z)$ can be realised as an orientation-preserving homeomorphism of the torus, so 
$\G_i$ is the fundamental group of a torus bundle over the circle.  In fact, it is a lattice in the Lie group  {\rm{{Sol}}},
which is one of the eight geometries in dimension 3 (see \cite{scott}).  Thus we have examples of closed,  geometric
3-manifolds that are not profinitely rigid. 
These types of torus-bundle examples were explored further by Funar \cite{funar}.   The corresponding punctured-torus
manifolds are more rigid \cite{BRW}.

\item Hempel \cite{hempel} produced further pairs of geometric 3-manifolds   with $\pi_1M_1\not\cong\pi_1M_2$
but $\wh{\pi_1M}_1\cong\wh{\pi_1M}_2$. These examples, which have the local geometry of $\H^2\times\R$, are higher-genus
surface bundles over the circle, where the holonomy has finite order. They are very close in spirit to Baumslag's finite-by-cyclic
examples in (2).
             
\item Following the work of 
Platonov and Tavgen \cite{PT} discussed in Section \ref{s:groth},  one knows that the direct product of two
or more non-abelian free groups is not profinitely rigid. 
\end{enumerate}

A noticeable absentee from the list of groups considered above is the free group $F_n$.  The following
question \cite[Qu.15]{probs} is a guiding challenge in the field.

\begin{question}[Remeslennikov] \label{q:rem}
Are finitely generated free groups profinitely rigid?
\end{question}

\subsection{Relative Profinite Rigidity and Hyperbolic Geometry}
There is now a large literature on relative profinite rigidity  with many positive and negative results, some of which 
are referenced in Section \ref{s:3M}  and subsections 5.2 and 6.2.  It is a very active field.  
I highlight the following result because it marked the beginning of a flowering of interest in the profinite
rigidity of low-dimensional orbifold groups, and also because it underscores the point I was making in Remark \ref{lemon-squeezy}:
this theorem shows that free groups are profinitely rigid within lattices, but  Question \ref{q:rem} is a vastly more difficult question
because it is hard to extract any geometry from the hypothesis $\wh{\Lambda}\cong\wh{F}_n$.

\begin{theorem}[Bridson, Conder, Reid \cite{BCR}] If  {$\G$} is a lattice in ${\rm{PSL}}(2,  \mathbb{R})$  
and $\Lambda$ is a lattice in any connected Lie group,   then  $\wh{\Gamma}\cong \wh{\Lambda} \implies \G\cong\Lambda$.
\end{theorem} 

The profinite invariance of the first $\ell_2$-Betti number plays an important role in the proof of this theorem.  In
particular, the work of Lott and L\"{u}ck \cite{LoLu} allows one to reduce fairly quickly to the case where $\Lambda$ is
a lattice in ${\rm{PSL}}(2,\R)$, where one understands the lattices very explicitly. 

We shall discuss the profinite rigidity of lattices in ${\rm{PSL}}(2,  \mathbb{C})$ in a moment. 
Stover \cite{stover1, stover2} showed that there are lattices in other rank-1 Lie groups 
(isometries of complex hyperbolic space) that are not profinitely rigid.

\section{Grothendieck Pairs}\label{s:groth}

The study of Grothendieck pairs was initiated by Alexander Grothendieck in 1970 \cite{groth}. He was 
investigating the extent to which the topological fundamental group of a smooth complex projective
 variety is determined by the \'{e}tale fundmental group of the underlying scheme,  and was interested in the extent to which a
 finitely generated,  residually finite group can be recovered from a complete knowledge of its finite dimensional
 representation theory.  
 
 \begin{definition}[Grothendieck Pairs]  A monomorphism of residually finite groups  $u:P\hookrightarrow \G$
is called a  {\em{Grothendieck pair}} if $\hat{u}: \wh{P}\to\wh{\G}$ is an isomorphism but $u$ is not.
One says that a finitely generated,  residually finite group $\G$ is 
{\em{Grothendieck rigid}}\footnote{There is also interest in fixing $\G$ and considering Grothendieck pairs
$\G\hookrightarrow\Lambda$.  When pursuing this,  it is natural to distinguish between the two settings with the terminology left/right (or up/down) Grothendieck rigidity \cite{owen, JL}. } if there are no Grothendieck pairs $P\hookrightarrow\G$ with $P\neq \G$ finitely generated. 
\end{definition} 

 Grothendieck 
 discovered a remarkably close connection between the representation 
theory of a finitely generated group and its
profinite completion:  if $A\neq 0$ is a commutative ring and 
$u:\G_1\to\G_2$ is a homomorphism of finitely
generated groups,  then 
$\hat u:\hat\G_1\to\hat\G_2$ is an isomorphism if and only 
the restriction functor $u_A^* :\text{Rep}_A(\G_2)
\to \text{Rep}_A(\G_1)$ is an equivalence of categories,
where  $\text{Rep}_A(\G)$ is the category of finitely
presented $A$-modules with a $\G$-action. 
This led him to pose his famous question about the existence of finitely presented Grothendieck pairs.

\begin{question}[Grothendieck 1970 \cite{groth}] If $\G_1$ and $\G_2$  are
 finitely presented and residually finite,   must
 $u:\G_1\to \G_2$ be an isomorphism if
 $\hat u :\hat\G_1\to\hat\G_2$ is an
isomorphism? 
\end{question}

Platonov and Tavgen \cite{PT} answered the corresponding question for finitely generated groups,  and various
other constructions followed \cite{pyber, BL, nek}.   Grothendieck's question was eventually settled 
by myself and  Fritz Grunewald \cite{BG}.  The following is an enhancement of what is stated in \cite{BG} -- see
Remark \ref{r:post-facto}.

\begin{theorem}[Bridson, Grunewald \cite{BG}]\label{thm:BG}
There exist (Gromov) {hyperbolic groups} {$\G<{\rm{SL}}(n,  \mathbb Z)$}
 and  finitely presented subgroups  $P\hookrightarrow G:=\G\times \G$ of infinite index,  
 with $P\not\cong G$,    such that $u:P\hookrightarrow G$ induces an
isomorphism $\hat u:\wh{P}\to \wh{G}$.  
\end{theorem}
 
\subsection{The ideas in  the proof}
A natural place to begin a search for Grothendieck pairs is with short exact sequences 
\begin{equation}\label{ses}
1\to N\to \G\to Q\to 1
\end{equation} 
with $N$ and $\G$ finitely generated and residually finite, $\wh{Q}=1$ and $Q$ an 
infinite group with no non-trivial finite quotients, e.g.  $G_4$ from (\ref{e:G_i})
or $B$ from  (\ref{e:B}).  If profinite completion were an exact functor,  we could simply
put hats on the groups in (\ref{ses}) to conclude that $\wh{N}\to\wh{\G}$ was an isomorphism.  But it is not an exact functor and 
while $\wh{N}\to\wh{\G}$ will be surjective,  it need not be injective.  To understand why, the reader
will need some familiarity with the (co)homology of groups \cite{brown}. Consider the 
5-term exact sequence in homology associated to (\ref{ses})
$$
H_2(Q,\Z)\to H_0(Q,\, H_1(N,\Z)) \to H_1(\G,\Z) \to H_1(Q,\, \Z)\to 0.
$$
As $N$ is finitely generated,  the automorphism group of its abeliansation $H_1(N,\Z)$ is residually finite,
so the action $Q\to {\rm{Aut}}(H_1(N,\Z))$ must be trivial,  since $\wh{Q}=1$,  and the group of coinvariants 
$H_0(Q,\, H_1(N,\Z))$ is simply $H_1(N,\Z)$.   Moreover,  $\wh{Q}=1$ implies $H_1(Q,\, \Z)=0$,  so (\ref{ses})
simplifies to an exact sequence
$$
H_2(Q,\Z)\to H_1(N,\Z) \to H_1(\G,\Z) \to 0,
$$
from which we deduce that if $H_2(Q,\Z)\to H_1(N,\Z)$ is not the zero map,  then $N$ will have finite abelian
quotients that $\G$ does not have,  whence $\wh{N}\not\cong\wh{\G}$.
To avoid this difficulty,  
we assume $H_2(Q,\Z)=0$,  a condition satisfied by our examples $G_4$ and $B$.  

\begin{lemma}\label{l:basic} 
With $1\to N\to \G\to Q\to 1$ as above,  if $H_2(Q,\Z)=0$ then $N\hookrightarrow\G$
is a Grothendieck pair.
\end{lemma}

With this lemma in hand,  we will have finitely generated Grothendieck pairs if we can construct short exact sequences
of the desired type.  To do this,  we appeal to the Rips construction \cite{rips} which,  given a finite presentation of a 
group $Q$,  will produce a short exact sequence $1\to N\to \G\to Q\to 1$,  as desired,  with $N$ finitely generated and
$\G$ a 2-dimensional hyperbolic group that satisfies a small cancellation condition; we saw in (\ref{ss:rf}) that $\G$ is residually finite.
To obtain further examples, one can verify that in this situation, if $H<\G$ is any subgroup that contains $N$
and if $\wh{H/N}=1$, then $N\hookrightarrow H$ and $H\hookrightarrow\G$ are both Grothendieck pairs. 
By taking $Q$ to be one of the universal groups constructed in  \cite[Theorem A]{mb:baum},  we deduce the following result.
\begin{theorem}\label{inf-many} 
There exist residually finite,  (Gromov) hyperbolic groups $\G$ of dimension $2$ with uncountably
non-isomorphic subgroups $\iota_H: H\hookrightarrow\G$ such that $\hat{\iota}_H:\wh{H}\to\wh{\G}$ is an isomorphism.
Moreover,  one can arrange for infinitely many of these subgroups to be finitely generated.
\end{theorem} 

This is not how Platonov and Tavgen \cite{PT} produced the first examples of finitely generated Grothendieck pairs.
(A version of the Rips construction that guaranteed residual finiteness was not available at that time.) Instead,
they considered {\em fibre products},  modifying the proof of Lemma \ref{l:basic} to establish the following
criterion,  which they applied to $F_4\to G_4$ to construct a Grothendieck pair $P\hookrightarrow F_4\times F_4$
with $P$ finitely generated.

\begin{proposition}[Platonov, Tavgen \cite{PT}]\label{p:PT}
Let $f:\G\to Q$ be an epimorphism of groups,  with $\G$ finitely generated and $Q$ finitely presented. 
Suppose that $\widehat{Q}=1$ and $H_2(Q,\Z)=0$.
Then,  the fibre product 
$P=\{(g,h) \mid f(g)=f(h)\} < \G\times \G$ is finitely generated and   
$P\hookrightarrow \G\times \G$   
induces an isomorphism $\widehat{P}\overset{\cong}\to\widehat{\G\times \G}$. 
\end{proposition} 
It is easy to see that $P$ is finitely generated: if $\G$ is generated by $X$ and $Q=\<X\mid R\>$
then $P$ is generated by $\{(x,x)\mid x\in X\}\cup \{(r,1) \mid r\in R\}$.  This fact is often called the
0-1-2 Lemma.  
But if $\G$ is free and $Q$ is infinite,  the fibre product $P$ will never have a finite presentation.  Indeed,
while the Platonov-Tavgen construction produces finitely generated Grothendieck pairs 
$P\hookrightarrow F_4\times F_4$ (and one can replace $F_4$ by $F_2$ using \cite{OS}),
the following theorem implies that there are no such finitely presented 
Grothendieck pairs,  even if one allows the product of more free groups,  or other residually free groups,  such as
surface groups. 

\begin{theorem}[Bridson, Wilton \cite{BW:closed}]\label{t:BWclosed}
If $\Pi$ is finitely presented and residually free,  then every finitely presented subgroup $P<\Pi$ is closed in the
profinite topology; in particular,  $P\hookrightarrow\Pi$ is not a Grothendieck pair.
\end{theorem}

\begin{remark}
Once more,  I want to emphasize the contrast in behaviour between  finitely presented subgroups and 
those that are merely finitely generated.
\end{remark}

Theorem \ref{thm:BG} is proved by combining the ideas described above and appealing to results on the finite
presentability of fibre products.  The {\em 1-2-3 Theorem} of Baumslag,  Bridson,  Miller and Short \cite{bbms}
states that if $Q$ has a classifying space with a finite 3-skeleton (as our examples $G_4$ and $B$ do),
while $\G$ is finitely presented and $N=\ker(\G\to Q)$  is finitely generated,  then the fibre product
$P<\G\times\G$ of $\G\to Q$ is finitely presented.  
To prove Theorem \ref{thm:BG} one can take $Q=B$,  use the Rips construction to generate $1\to N\to \G\to B\to 1$,
and then appeal to Wise's work \cite{wise-sc} to see that $\G$ is a cubulated hyperbolic group,  which by Agol's Theorem \cite{agol}
can be embedded in ${\rm{SL}}(d,\Z)$ for some integer $d$.  The 1-2-3 Theorem assures us that the
associated fibre product $P$ is finitely presented,  and Proposition \ref{p:PT} assures us that  
 $P\hookrightarrow\G\times\G$ is a Grothendieck pair.
 
\begin{remark}\label{r:post-facto} When \cite{BG} was written,  neither Agol's theorem
nor the virtually special Rips construction in \cite{HW} were available,
so the penultimate sentence of the proof sketched above was replaced 
by an appeal to Wise's residually finite version of the Rips construction \cite{wise:rf}. 
This accounts for the difference between Theorem \ref{t-intro:BG} and Theorem \ref{thm:BG}.
\end{remark} 
   
\subsection{Elaborations,  Refinements and Grothendieck rigidity}

In the last twenty years,  many papers have been written constructing Grothendieck pairs $P\hookrightarrow\G$
where $P$ and $\G$ have different properties.  Indeed,  this has been perhaps the main mechanism by which various
properties have been shown not to be profinite.  Typically,  this is achieved
by encoding additional properties into a group $Q$ that has  a finite classifying space
  with $\wh{Q}=1$ and $H_2(Q,\Z)=0$.  One then
  follows  the above template of proof, perhaps adjusting $\G$ in some way,  by
refining the Rips construction,  or passing to central extension \cite{mb:baum},  or replacing the hyperbolic group $\G$ by something entirely different \cite{mb:FbyF}).  Often,  a more refined version of the 1-2-3 Theorem is needed,  
as in \cite{mb-jems} where it is shown that there exist finitely presented,  residually finite groups $\G$ that
contain infinitely many non-isomorphic,  finitely presented subgroups $P_n$ such that $P_n\hookrightarrow\G$
is a Grothendieck pair.

In the opposite direction,  there are notable results proving that certain groups are Grothendieck rigid.  
Grothendieck himself \cite{groth} proved this for groups of the form ${\rm{G}}(A)$ where $A$ is a commutative
ring and ${\rm{G}}$ is an affine group scheme of finite type over $A$ 
  -- see Jaikin-Zapirain and Lubotzky \cite{JL} for a concise account of this and further results.  In the setting of 3-manifolds, 
Long and Reid  \cite{LR} established Grothendieck 
rigidity for  the  fundamental groups of all closed geometric $3$-manifolds and
all  finite volume hyperbolic $3$-manifolds.  Following the work of Agol and Wise,   
Sun was able to extend this to cover the fundamental group of every compact $3$-manifold  \cite{Sun}. 
In the hyperbolic case,  Theorem A of \cite{prasad}
 proves something stronger,  ruling out abstract isomorphisms 
$\wh{H}\cong\wh{\G}$  for subgroups $H<\G$, 
not just those induced by inclusions $H\hookrightarrow\G$.  This stronger
statement is false in the non-hyperbolic setting \cite{hempel,  funar}.

\section{Profinite rigidity and 3-manifolds}\label{s:3M}

Thurston's Geometrisation Conjecture, proved by Pereleman, tells us that every closed orientable 3-manifold can be cut into pieces by
a system of embedded spheres and tori so that the interior of each piece is homeomorphic to a 
finite-volume quotient on one of the eight 3-dimensional geometries \cite{scott}. 
Hyperbolic geometry $\H^3$  is the richest and most ubiquitous of these geometries
and we encountered {\rm{Sol}} in (\ref{ss:sobering}(4)).
The pieces modelled on the remaining six geometries are Seifert fibre spaces,  which we shall discuss in a moment.

Many beautiful results have been proved in recent years concerning the great extent to which properties 
of 3-manifolds are captured by the profinite completions of their fundamental results.   {\em{As a result, we now
know that only finitely many compact 3-manifolds $M$ can have the same profinite completion, and we know that many features
of the 3-manifold are captured by $\wh{\pi_1M}$. }}

A 3-manifold  is  {\em irreducible} if every 
2-sphere in it bounds a ball.  Every closed, orientable 3-manifold can be written as a connected sum of irreducible manifolds
and copies of $S^2\times S^1$.  In the spirit of the Geometrisation Conjecture,  the first thing that one wants to know is
whether the prime decomposition of $M$  is visible in  $\wh{\pi_1M}$, and
 Wilton and Zalesskii \cite{WZ3} proved that it is: if $M$ and $N$ are closed, orientable 3-manifolds with prime decompositions
 $M=M_1\#\cdots\# M_m\# r(S^1\times S^2)$ and $N=N_1\#\cdots\# N_n\# s(S^1\times S^2)$, then $m=n,\ r=s$, and 
 after permuting the indices,  $\wh{\pi_1M_i}\cong \wh{\pi_1N_i}$.  
 Wilkes \cite{wilkesJPAA} showed that this remains valid if $M$
 has incompressible boundary.
These results  allow us to concentrate on irreducible manifolds,
where attention falls on the process of cutting the manifold along tori.   In this context,  Wilton and Zalesskii \cite{WZ3}   
showed that $\wh{\pi_1M}$ encodes the
the canonical  JSJ decomposition of the manifold, including the profinite completions of the pieces into which the tori
cut the manifold; again, Wilkes shows that one can allow boundary \cite{wilkesNZ}.
Wilton and Zalesskii also proved that the profinite completion distinguishes the hyperbolic pieces in the geometric decomposition from the Seifert fibred pieces \cite{WZ1}.

We saw in (\ref{ss:sobering}(4)) that there are distinct quotients
of ${\rm{Sol}}$ that have the same profinite completions,  but there can be only finitely many with a give completion \cite{GPS}. 
The same is true of $\H^2\times\R$ manifolds.  Wilkes \cite{wilkes:sfs} extended this finite genus result to manifolds that consist entirely
of Seifert pieces, giving a complete account of the ambiguities. 

The first example of a hyperbolic 3-manifold that is distinguished from all other 3-manifolds by $\wh{\pi_1M}$ was the
figure-8 knot complement (Boileau-Friedl \cite{BF} and Bridson-Reid \cite{BR:fig8}).   
Bridson, Reid and Wilton later proved that every punctured torus bundle over the circle 
is distinguished among 3-manifold groups by $\wh{\pi_1M}$. (Punctured torus bundles account for all groups of the form
$F_2\rtimes\Z$. 
It remains unknown whether groups of the form $F_n\rtimes\Z$ are distinguished from each other by  
their profinite completions -- cf.~\cite{HK,  BP}.) These early results relied on recovering fibre-bundle structure from $\wh{\pi_1M}$. 
Subsequently,  Jaikin-Zapirain \cite{jaikin:fib}  proved that $\wh{\pi_1M}$ always detects whether a 3-manifold fibres over the circle with compact fibre.  Other notable results include the fact that 3-manifold groups are all
good in the sense of Serre \cite{serre:CG, wilkes-book,  reid:stA},  i.e.~if $\G$ is a 3-manifold group, then for all $n>0$ and all finite $\G$-modules $A$, the natural map $H^n(\wh{\G}, A)\to H^n(\wh{G}, A)$ from the continuous cohomology of $\wh{\G}$
to the group cohomology of $\G$, is an isomorphism.  We refer to Reid's ICM talk \cite{reid:icm} for an overview of
results up to 2018 and \cite{reid:kias} for updates.

A  breakthrough that deserves particular attention comes from the work of
Liu \cite{liu}.  He proved that only finitely many finite-volume hyperbolic 3-manifolds can have the same profinite completion.
Building on much of the work described above, 
Xiaoyu Xu recently extended this finite genus theorem to all compact, orientable,  3-manifolds with toral or empty boundary \cite{xu}.
\smallskip

The following discussion of Seifert fibre spaces is needed for Theorem \ref{t-intro:BRS}.

\subsection{Seifert Fibre Spaces}  
A  {\em Seifert fibre space} is a 3-manifold that is foliated by circles.  The circles of the foliation are divided into  regular fibres and  exceptional fibres: a circle is a regular fibre if it has a neighbourhood that admits a fibre-preserving homeomorphism to the solid torus $\mathbb{D}^2\times S^1$ foliated 
by the circles $\{x\}\times S^1$; in a neighbourhood of an exceptional fibre,  the foliation is modelled on the foliation
of  $\mathbb{D}^2\times S^1$ that is obtained by cutting the product fibration open along a disc $\mathbb{D}^2\times \{*\}$ and then regluing after twisting through an angle $2q\pi/p$,  where $p,q\in \Z$ are coprime integers with  $0 < q < p$. 
By collapsing each circle fibre in $M$ to a point,  
we obtain a $2$-manifold with an orbifold structure given by marking each exceptional
fibre with a cyclic group recording the order $p$ of the twisting there; this is called the base orbifold of the Seifert fibration.

We will be concerned only with Seifert fibred manifolds  where $\pi_1M$ is infinite,  in which
case $\pi_1(M)$ fits into a central extension 
$$ 1 \to \Z \to \G \to \Delta \to 1,$$
where the central $\Z$ is the fundamental group of a regular fibre and $\Delta$ is the fundamental group of 
the base orbifold.  For example,   if $M$ has 3 exceptional fibres,  with twisting data $(p_1,q_1),\, (p_2,q_2),\,
(p_3,q_3)$ and the underlying surface of the base orbifold is a sphere,  then $\Delta = \Delta(p_1,p_2,p_3)$ and
$$
\pi_1(M) = \< c_1, c_2,c_3, z\mid z~\hbox{is central},~c_1^{p_1}z^{-q_1}=c_2^{p_2}z^{-q_2}=c_3^{p_3}z^{-q_3}=1, c_1c_2c_3 = z^d\>
$$
for some $d\in\Z$. 

The examples of Hempel  (\ref{ss:sobering}(5)) are surfaces bundles over the circle with finite holonomy.
Wilkes \cite{wilkes:sfs} proved that all Seifert manifolds apart from these are rigid.

\section{Absolute Profinite Rigidity}

We seek finitely generated,   residually finite groups $\G$ that are profinitely rigid in the absolute sense.   
We observed in (\ref{ss:sobering}(3)) that if $\G$ satisfies a group law, then this immediately reduces to a 
problem within the class of groups that satisfy this law -- i.e. ~it reduces to a challenge concerning relative profinite
rigidity.   But we also  alluded (Remark \ref{lemon-squeezy})  to the difficulty of  knowing where to begin if 
there is no obvious reduction to a relative problem.  That is the fundamental challenge that we want to address 
in this section.  To avoid simple reductions, we consider only groups that are {\em{full-sized}}, meaning that 
they contain a non-abelian free group. {\em{Do there exist finitely generated, residually finite, full-sized groups that are
profinitely rigid?}}  
Might 3-manifolds and  hyperbolic geometry\footnote{Always remember,   ${\mathbb{H}}$yperbolic geometry  is good for you!}
help us?

\subsection{Profinitely Rigid Kleinian Groups}

 We begin by describing some hyperbolic 3-orbifolds of small volume.  
The reader will likely be familiar with the tiling of the hyperbolic plane by ideal triangles. In the same way, one can
tile hyperbolic 3-space $H^3$ with copies of the regular ideal tetrahedron, i.e.~the tetrahedron
with vertices at infinity where the dihedral angle between each pair of faces is $\pi/3$.  Barycentric subdivision
divides each tetrahedron in the tiling  into 24 congruent tetrahedra, each with one vertex at infinity. The reflections in 
the faces of any one of these smaller tetrahedra will generate the full isometry group of the tiling. This group of
isometries is
$$\Lambda_0=\< A,   B,   C,   D \mid A^2,   B^2,   C^2,   D^2,   (AB)^3,   (AC)^3,   (AD)^2,   (BC)^3,   (BD)^2,   (CD)^6 \>.$$ 
This  is the unique lattice of smallest covolume in ${\rm{Isom}}(\H^3)$.
It contains the Bianchi group ${\rm{PSL}}(2,  \Z[\omega])$  as a subgroup of index 4,  
where $\omega^2+\omega+1=0$.  

Among  torsion-free uniform  lattices in  ${\rm{PSL}}(2,\C)$ there is a unique one of smallest covolume, 
 namely $\G_W$, the fundamental group of the Weeks manifold \cite{weeks}, which can be obtained by Dehn surgery on the Whitehead link.

\begin{theorem}[Bridson  McReynolds,  Reid,  Spitler \cite{BMRS1}]\label{thm:bmrs}
Each of   $\G_W,  \ \Lambda_0,  \ {\rm{PSL}}(2,  \Z[\omega]), $ and $ {\rm{PGL}}(2,  \Z[\omega])$
is { {profinitely rigid}} among all finitely generated,   residually finite groups.   
\end{theorem}
The proofs in \cite{BMRS1} cover a number of related lattices, but only finitely many. The list of lattices in ${\rm{PSL}}(2,\C)$
that are known to be profinitely rigid (in the absolute sense) was extended in \cite{prasad} and \cite{CW} but remains finite
for the moment.

The following are some of the key elements of the proof of Theorem \ref{thm:bmrs} for the given examples $\G$.
\begin{itemize}[leftmargin=*]
\item Arithmetic: The trace field $K_\G$ of $\G$ (defined below) has small degree and $\G$ is commensurable with the group
of units of a maximal order in the associated quaternion algebra.  
\item Character varieties: $\G$ has very few Zariski-dense representations into ${\rm{PSL}}(2,\C)$, up to conjugacy,  and
these are all accounted for by the arithmetic of $K_\G$ ({\em{``Galois rigidity"}}).
\item Mostow rigidity and volume calculations are useful,
 as are various elements of 3-manifold topology, including the existence of
compact cores,  some knot theory, and the existence of finite covers that fibre.
The consequences of cubulation,  following Agol \cite{agol} and Wise \cite{wise:qc},  play a crucial role,
not least in ensuring  cohomological goodness in the sense of Serre \cite{serre:CG}.
\item Intricate arguments concerning  the subgroups of small finite index in each example were needed  in \cite{BMRS1} but 
these were superceded by the results in \cite{prasad}.
\end{itemize}

\subsection{Trace fields and Galois rigidity} \label{s:Galois}

A remarkable feature of hyperbolic 3-manifolds is that their fundamental groups
have  {\em{arithmetic structure}} intrinsically associated to them -- see Maclachlan and Reid \cite{MR}.
Much of the proof of  Theorem \ref{thm:bmrs} revolves around this arithmetic structure and the way it controls  representations
into ${\rm{PSL}}(2,\C)$ in our examples.  A key notion in this regard is {\em Galois rigidity}.   I will explain very briefly what this is and how it is used.

Let $\phi\colon \mathrm{SL}(2,\mathbb{C}) \to \mathrm{PSL}(2,\mathbb{C})$ be the quotient homomorphism, and if $H$ is a finitely generated subgroup of $\PSL(2,\mathbb{C})$,  set $H_1 = \phi^{-1}(H)$. It will be convenient to say that $H$ is 
{\em Zariski-dense} in $\PSL(2,\C)$ when what we actually mean is that $H_1$ is a Zariski-dense subgroup of $\SL(2,\C)$. The \textit{trace-field} of $H$ is defined to be the field 
\[ K_H=\mathbb{Q}(\mathrm{tr}(\gamma)~\colon~ \gamma \in H_1). \] 
If $K_H$ is a number field with ring of integers $R_{K_H}$, we say that $H$ has {\em integral traces} if $\tr(\gamma)\in R_{K_H}$ for all $\gamma \in H_1$.  

An important object associated to $H$ is the subalgebra of ${\rm{Mat}}(2,\C)$ generated by $H_1$; this is 
a {\em quaternion algebra} over $K_H$.

Suppose now that $\La$ is an abstract finitely generated group and  $\rho\colon \La\rightarrow \PSL(2,\C)$ a Zariski-dense representation with $K=K_{\rho(\La)}$ a number field of degree $n_K$. If $K=\Q(\theta)$ for some algebraic number $\theta$, then the Galois conjugates of $\theta$, say $\theta=\theta_1,\dots,\theta_{n_K}$, provide embeddings $\sigma_i\colon K\to\C$ defined by $\theta\mapsto\theta_i$.  These in turn can be used to build $n_K$ Zariski-dense non-conjugate representations $\rho_{\sigma_i}\colon \La \to \PSL(2,\C)$ with the property that $\tr(\rho_{\sigma_i}(\gamma))=\sigma_i(\tr\rho(\gamma))$ for all $\gamma\in \La$. 
One refers to these as {\em Galois conjugate representations}. 
The existence of these Galois conjugates shows that $|\mathrm{X}_{\mathrm{zar}}(\G,\mathbb{C})|\geq  n_{K_{\rho(\La)}}$, where 
$\mathrm{X}_{\mathrm{zar}}(\La,\C)$ denotes the set of Zariski-dense representations $\G\to\PSL(2,\C)$ up to conjugacy.

\begin{definition}\label{def:galois-rigid}
Let $\La$ be a finitely generated group and $\rho\colon \La\to \PSL(2,\C)$ a Zariski-dense representation whose trace field $K_{\rho(\La)}$ is a number field. If
 $|\mathrm{X}_{\mathrm{zar}}(\La,\mathbb{C})|= n_{K_{\rho(\G)}}$, we say that $\G$ is {\em Galois rigid}  (with associated field $K_{\rho(\La)}$).
\end{definition}  
 
\subsection{An outline of the proof of Theorem  \ref{thm:bmrs}}

An explicit understanding of the   character variety of each of our examples enables us to prove
that each is Galois rigid. For example, we prove that the only Zariski dense representations of ${\rm{PSL}}(2,\Z[\omega])$,
whose trace field is $\Q(\sqrt{-3})$, are the inclusion map and its complex conjugate. In each case, one also has integrality of traces. The integrality of traces is used to ensure the boundedness of the non-archimedean representations  
$\G\to {\rm{SL}}(2,\-{\Q_p})$ obtained by  transporting Zariski dense representations 
$\G\to {\rm{SL}}(2,\C)$  via (non-continuous) field isomorphisms $\C\cong\-{\Q_p}$, for all primes $p$. This boundedness enables one to extend the representations to $\wh{\G}$,  then restrict to any finitely generated $\Lambda$ with
$\wh{\Lambda}\cong\wh{\G}$ to obtain bijections between the sets of bounded, Zariski-dense  characters of $\G$ 
and $\Lambda$ at every finite place (Lemma 4.6 of \cite{BMRS1}). 

With this control established, one can argue
that the abstract group $\Lambda$ is Galois rigid with a Zariski-dense representation $\rho\colon \La\to \PSL(2,\C)$
whose arithmetic data match those  of $\G$ -- see \cite[Theorem 4.8]{BMRS1} for a precise statement. In good situations
(which include the low-degree number fields of our examples),  this matching of data forces the number fields
$K_{\rho(\Lambda)}$ and $K_\G$ to be equal,  and likewise the quaternion algebras of $\G$ and $\rho(\Lambda)$.
 This in turn allows one to
conclude that (up to Galois conjugation) the image of $\rho:\La\to  \PSL(2,\C)$ is contained in $\G$ or a small
extension of it, and in all of the examples one can force the image to lie in $\G$. 

At this point of the argument in 
\cite{BMRS1}, a host of results about subgroups of 3-manifolds in general, and the finite-index subgroups of our
examples in particular, were used to argue that if the image of $\rho$ was not equal to $\G$ then $\La$ would have
a finite quotient that $\G$ did not have. These arguments can now be replaced by \cite[Theorem A]{prasad}: {\em{
if $\Pi<{\rm{PSL}}(2,\C)$ is a lattice and $H<\Pi$ is a finitely generated, proper subgroup, then 
$\wh{H}\not\cong\wh{\Pi}$; if $H$ has finite index in $\Pi$, then there is a finite group onto which $H$ maps but $\Pi$ does not. }} 

At this stage, we have proved that there is a surjection $\rho:\La\twoheadrightarrow\G$, hence $\wh{\rho}:\wh{\La}
\twoheadrightarrow\wh{\G}$, and the Hopf property for finitely generated profinite groups (a variant of
Lemma \ref{l:hopf}) enables us to conclude that $\wh{\rho}$ is injective and hence so is $\rho$.  \qed

\subsection{Profinitely rigid Fuchsian groups} 
Given integers $(p,q,r)$ with $1/p + 1/q + 1/r < 1$, one can tile  the hyperbolic plane with triangles whose interior
angles are $2\pi/p,\,  2\pi/q,\,   2\pi/r$.  The group  generated by reflections is the sides of a fixed triangle acts
transitively on the set  of triangles in the tiling; we denote this group $\Delta^{\pm}(p,q,r)$. The index-2
subgroup consisting of orientation-preserving isometries is $\Delta(p,q,r)$.  We restrict our attention to 
the following triples $(p,q,r)$: 
\begin{equation}\label{list-top}
(3,3,4), \ \ (3,3,5),\ \ (3,3,6),\ \ (2,5,5), \ \ (4,4,4).   
\end{equation} 
\begin{equation}\label{list-bottom}
\ \ (2,3,8), \ \ (2,3,10),\ \ (2,3,12),\ \ 2,4,5), \ \ (2,4,8).
\end{equation}

\begin{theorem} [Bridson,  McReynolds,  Reid,  Spitler \cite{BMRS2}] 
Both $\Delta(p,q,r)$ and  $\Delta^{\pm}(p,q,r)$ are profinitely rigid if $(p,q,r)$ belongs to one of these lists
\end{theorem} 
A key fact about these groups is that they are Galois rigid (as subgroups of ${\rm{PSL}}(2,\C)$) and in each case
the trace field is a real quadratic number field where the ramification of primes is simple enough to ensure that
the argument from \cite{BMRS1} sketched above applies.    

\bigskip

{ {\em We have discovered some lattices in ${\rm{PSL}}(2,\R)$
and  ${\rm{PSL}}(2,\C)$  that are profinitely rigid in the absolute sense, but conjecturally there should be many more.}}
 
 \section{Seifert Fibre Spaces and the importance of finiteness properties}

The following theorem provides the first examples of finitely presented groups that are profinitely rigid among
finitely presented groups but not among finitely generated groups.   Like our
earlier examples,  the groups $\G$ in this theorem are 3-manifold groups with particular arithmetic properties. 
Here, $S^2(p,q,r)$ denotes the quotient $\mathbb{H}^2/\Delta(p,q,r)$ of the hyperbolic plane by the 
 triangle group $\D(p,q,r)$.
 
\begin{theorem}[Bridson,  Reid,  Spitler \cite{BRS}] \label{t:duke} 
There exist finitely presented, residually finite groups $\G$ with the following properties:
\begin{enumerate}[leftmargin=*]
\item $\G\times\G$ is profinitely rigid among all finitely presented, residually finite groups. 
\item There exist infinitely many non-isomorphic finitely generated groups $\Lambda$ such that  $\wh{\Lambda}\cong \wh{\G}\times\wh{\G}$.
\item If $\Lambda$ is as in (2), then there is an embedding
$\Lambda\hookrightarrow\G\times\G$ that induces the isomorphism $\wh{\Lambda}\cong \wh{\G\times\G}$
(in other words,  $\Lambda\hookrightarrow\G\times\G$ is a Grothendieck pair). 
\end{enumerate}  
If $M$ is any Seifert fibred space 
with base  orbifold $S^2(3,3,4)$ or $S^2(3,3,6)$ or $S^2(2,5,5)$,  then $\G=\pi_1M$ has these properties. 
\end{theorem}

\subsection{An outline of the Proof}
We consider Seifert fibre spaces $M$
whose base orbifold is $S^2(p,q,r) = \H^2/\D(p,q,r)$, where $(p,q,r)$ is one of triples from   list (\ref{list-top}) or
(\ref{list-bottom}).  We know from the previous section that $\Delta=\Delta(p,q,r)$ is profinitely rigidity.  There is a central extension
$$
1\to\Z\to \pi_1M \to \Delta\to 1.
$$
\begin{theorem}\label{SFS_rigid} 
Let $\Delta(p,q,r)$  be a triangle group from  list (\ref{list-top}) or (\ref{list-bottom}), and let $M$ be a Seifert fibred space with base $S^2(p,q,r)$.  
Then $\pi_1 M$ is profinitely rigid (in the absolute sense).
\end{theorem}   
In order to prove this theorem, one has to extend the arguments about profinite rigidity in the previous sections
to cover central extensions (this is not hard). There are subtleties, however, and it is is not the case that  an arbitrary
central extension of $\Delta$ will be profinitely rigid \cite{piwek}. Special arguments apply in the case where the
centre is cyclic. Wilkes' work on the
profinite rigidity of Seifert fibred spaces within the class of 3-manifold groups  \cite{wilkes:sfs}  is used here.

The key step of reducing abstract profinite isomorphism to the study of Grothendieck pairs also requires
an extension of the scope of our arguments about Galois rigidity, this time to cover direct products 
and fibre products as well 
as central extensions.

\begin{theorem} \label{t3}
Let $\Delta(p,q,r)$ be a triangle group from list  (\ref{list-top}) or  (\ref{list-bottom}), let $M$ be a 
Seifert fibred space with  base orbifold $S^2(p,q,r)$ and let $\G=\pi_1M$.
Then, for  every finitely generated, residually
 finite group $\La$ with $\wh{\La}\cong\wh{\G\times\G}$, there is
 an embedding $\La\hookrightarrow\G\times\G$ that induces the isomorphism.
 \end{theorem}

 With this reduction to Grothendieck pairs in hand,   
 the following   straightforward extension of Theorem \ref{t:BWclosed} establishes the desired profinite rigidity 
 among finitely presented groups. 
\begin{theorem}\label{t:not-fp}
For every Fuchsian group $F$
and every Seifert fibred space  $M$ with base orbifold $\H^2/F$, there are no Grothendieck pairs $\La\hookrightarrow
\pi_1M\times\pi_1M$ with $\La\neq \pi_1M\times\pi_1M$ finitely presented.
\end{theorem}

To complete the proof of Theorem \ref{t:duke} we need finitely generated Grothendieck pairs.

\begin{theorem}\label{t:not-GR}
If $\Delta(p,q,r)$ is a triangle group from  list (\ref{list-top}), 
then there are infinitely many
Seifert fibred spaces with  base orbifold $S^2(p,q,r)$ whose fundamental group $\G$ has the 
property that there are infinitely many non-isomorphic, finitely generated groups $\La$
and inclusions $\La\hookrightarrow \G\times \G$ inducing isomorphisms $\wh{\La}\cong\wh{\G\times \G}$. 
 \end{theorem} 
 Note that at this stage,  not all of the groups from Theorem \ref{SFS_rigid} are used:
we need constraints on the Seifert invariants that ensure  
$\G=\pi_1M$ has finite index in $[G,G]$ where $G$ is a group that maps onto a non-elementary hyperbolic group and has $H_1(G,\Z)$ finite and $H_2(G,\Z)=0$.  
These conditions facilitate a homological argument, adapted from an idea of
Bass and Lubotzky \cite{BL}, that allows us to modify the template for constructing Grothendieck pairs that was
described in Section \ref{s:groth}.

\section{Open Questions} 

The most celebrated problem concerning profinite rigidity is  Remeslenikov's Question (\ref{q:rem}), where the answer
is widely believed to be positive: see \cite{mb:catriona} for a discussion of related results and recent progress,
most notably \cite{Jaikin-Zapirain:parafree}.
\begin{conjecture}
Finitely generated free groups are profinitely rigid.
\end{conjecture}

The consensus on whether the following conjecture is true is less uniform: it is widely believed
that lattices in ${{\rm{PSL}}}(2,  \R)$ are profinitely rigid,  but many would favour a more cautious conjecture
in the Kleinian case, simply asking that a finitely generated, residually finite group $\La$
with the same finite quotients as a  lattice in ${{\rm{PSL}}}(2,  \mathbb{C})$  should itself be a lattice.  
If we knew that $\La$ was a lattice in ${{\rm{PSL}}}(2,  \mathbb{C})$,
Liu's finite genus theorem \cite{liu} would tell us that there are only finitely many possibilities for $\La$; 
it is unclear if this finite ambiguity can be removed but recent results of \cite{liu} and \cite{W:if} constrain it.

\begin{conjecture}
All lattices in ${{\rm{PSL}}}(2,  \R)$ and ${{\rm{PSL}}}(2,  \mathbb{C})$ are profinitely rigid (in the absolute sense).
\end{conjecture}

With Theorem \ref{t:duke} in mind, one might weaken the above conjectures and ask only for rigidity 
among finitely presented groups.  We saw in Section \ref{s:groth} that the difference between finite generation and finite
presentation definitely matters is the context of direct products of free groups.   

\begin{conjecture} Let $F$ and $F'$ be finitely generated free groups. If a finitely presented residually free group
$\G$ has the same finite quotients as $F\times F'$, then $\G\cong F\times F'$.
\end{conjecture}

For groups of higher-dimensional origin, the most compelling question in the field is the following. 
\begin{question} Is ${\rm{SL}}(n,  \Z)$ profinitely rigid when $n\ge 3$?
\end{question}

Using results from Spitler's thesis \cite{spitler} that were inspired by ideas of Lubotzky (and are close in spirit
to our discussion of Galois rigidity), one can reduce this question to the challenge of deciding whether ${\rm{SL}}(n,  \Z)$ 
is Grothendieck rigid. But since the subgroups of ${\rm{SL}}(n,  \Z)$  are so varied and mysterious,  this remains a very difficult challenge.


\begin{thebibliography}{999} 


\bibitem{adian} S.I. ~Adian, {\em The unsolvability of certain algorithmic problems in the
theory of groups}, Trudy Moskov. Mat. Obsc.~{\bf{6}}(1957), 231--298.
 


\bibitem{agol} I.~Agol, {\em The virtual Haken conjecture}, with an
appendix by I. Agol, D. Groves, and J. Manning, Documenta Math. {\bf 18} 
(2013), 1045--1087.
   
   
\bibitem{aka1} M.~Aka, {\em Arithmetic groups with isomorphic finite quotients},
J. Algebra {\bf 352} (2012), 322--340.


\bibitem{aka2} M.~Aka, {\em Profinite completions and Kazhdan's Property T,}  Groups Geom.~Dyn.
{\bf{6}} (2012),  221--229. 


\bibitem{AFW} M.~Aschenbrenner, S.~Friedl and H.~Wilton, {\em 3-Manifold Groups},  Lecture Notes in Math., European Mathematical Society,  Z\"urich, 2015.

\bibitem{BL} H.~Bass and A.~Lubotzky,
{\em Nonarithmetic superrigid groups: counterexamples to Platonov’s
conjecture},  Annals of Math.{\bf{151}} (2000), 1151--1173.

\bibitem{BCG} I. Bauer, F. Catanese and F.~Grunewald,
{\em Faithful actions of the absolute Galois group on connected components of moduli spaces}, 
Invent.  Math.  {\bf{199}} (2015),  859--888.

\bibitem{Baum74} G.~Baumslag, {\em Residually finite groups with the same
finite images}, Compositio Math. {\bf 29} (1974), 249--252.

\bibitem{bbms}  
G.~Baumslag, M.R.~Bridson, C.F.~Miller III, H.~Short, \emph{Fibre products, 
non-positive curvature, and decision problems},
Comm.~Math.~Helv.  {\bf 75} (2000), 457--477.



\bibitem{BDH} G.~Baumslag,  E.~Dyer and  A.~Heller, 
{\em The topology of discrete groups},
J. Pure Appl. Algebra {\bf 16} (1980),   1--47.  

\bibitem{BaRo} G.~Baumslag and J.~Roseblade, {\em Subgroups of direct products of free groups}, J.~London Math.~Soc. {\bf 30} (1984), 44--52.

\bibitem{BestBrady}  M.~Bestvina and N.~Brady, 
{\em Morse theory and finiteness properties of groups.},
 Invent.~Math.~{\bf{129}} (1997),  445--470.

\bibitem{BNS} R.~Bieri,  W.D.~Neumann and R.~Strebel, 
{\em{A geometric invariant for discrete groups}}, Invent. Math. {\bf{90}}, (1987), 451--477.


\bibitem{BF} M. Boileau and S. Friedl, {\em The profinite completion of
3-manifold groups, fiberedness and the Thurston norm}, 
in What's Next?: The Mathematical Legacy of William P. Thurston, Annals of Math. Study {\bf 205}, 21--44, Princeton 
Univ.~Press, NJ, 2020.


\bibitem{boone} W.W. Boone, {\em Certain simple unsolvable problems in group theory}, I,
II, III, IV, V, VI, Nederl. Akad.Wetensch Proc. {\bf 57} (1954) 231--237, 492--497; {\bf 58} (1955) 252--256, 571--577; {\bf 60} (1957) 22-27, 227--232.


\bibitem{boris}
V. Borisov, {\em Simple examples of groups with unsolvable word problem}, 
Math. Zametki {\bf 6} (1969) 521--532; English transl., Math. Notes {\bf 6} (1969)  768--775.


\bibitem{mb-jems} M.R.~Bridson, 
{\em The strong profinite genus of a finitely presented group can be infinite},
 J. Eur. Math. Soc.  {\bf 18} (2016), no. 9, 1909--1918.

\bibitem{mrb:agt} M.R.~Bridson, 
{\em Profinite isomorphisms and fixed-point properties}, 
Alg.~Geom.~Topol.~{\bf{24}} (2024), 4103--4114.

\bibitem{mb:baum} M.R.~Bridson,  
{\em The homology of groups,  profinite completions,  and echoes of Gilbert
Baumslag},  in ``Elementary Theory of Groups and Group Rings,  and Related
Topics",  pp.~11--28,  De Gruyter,  Berlin,  2020. 

 
\bibitem{mb:raag} M.R.~Bridson, {\em On the recognition of right-angled Artin groups},
Glasgow Math. J.  {\bf{62}} (2020),  473--475.
 
 \bibitem{mb:h2} M.R.~Bridson, {\em The Schur multiplier, profinite completions and decidability,} Bull.~London
Math.~Soc.~{\bf{42}} (2010),  412--416.

\bibitem{mb:catriona} M.R.~Bridson, {\em Profinite rigidity and free groups}, in ``Mathematics Going Forward",
J-M.~Morel, B.~Teissier (eds.),  Lect. Notes Math, vol 2313,  pp.~233--240, Springer Verlag,  Basel. 


\bibitem{mb:FbyF} M.R.~Bridson, {\em Profinite completions of free-by-free groups contain everything},
Quart.~J.~Math., {\bf{75}} (2024),  139--142.

\bibitem{BCR} M. R.~Bridson, M.D.E.~Conder and A. W.~Reid, {\em Determining 
Fuchsian groups by their finite quotients},  Israel J. Math.  {\bf{214}} (2016),  1--41.

\bibitem{BELS}
M.R.~Bridson,  D.M.~Evans,  M.W.~Liebeck and D.~Segal,
{\em  Algorithms determining finite simple images of finitely presented groups},
Invent.  Math.  {\bf{218}} (2019),  623--648.


\bibitem{BG} M.R.~Bridson and F.~Grunewald, {\em Grothendieck's problems concerning profinite completions and representations of groups},
Annals of Math.~{\bf 160}  (2004), 359--373.



\bibitem{BH}
 M.R.~Bridson and A.~Haefliger,
``Metric Spaces of Non-Positive Curvature", Grund. Math. Wiss.
{\bf 319}, Springer-Verlag, Heidelberg-Berlin, 1999.
 

\bibitem{BHMS}
M.R. Bridson, J. Howie, C.F. Miller~III and H.~Short, 
{\em Subgroups of direct products of limit groups},  Annals of Math. {\bf 170} (2009),  1447--1467.


\bibitem{BMRS1} M.R.~Bridson, D.B.~McReynolds, A.W.~Reid and R.~Spitler, \emph{Absolute profinite rigidity and hyperbolic geometry}, Annals of Math. {\bf 192} (2020), 679--719.
 

\bibitem{BMRS2} M.R.~Bridson, D.B.~McReynolds, A.W.~Reid and R.~Spitler, \emph{On the profinite rigidity of triangle groups}, Bull. London Math. Soc. {\bf 53} (2021), 1849--1862.  

\bibitem{BP} M.R.~Bridson and P.~Piwek,
{\em Profinite rigidity for free-by-cyclic groups with centre},
{\tt{arXiv:2409.20513}}.


\bibitem{BR:fig8} M.R.~Bridson and A.W.~Reid, {\em Profinite rigidity, fibering and the figure-eight knot}, in What's Next?: The Mathematical Legacy of William P. Thurston, Annals of Math. Study {\bf 205}, 45--64, Princeton Univ.~Press, NJ, 2020.


\bibitem{prasad} M.~R.~Bridson, and A.~W.~Reid, {\em Profinite rigidity, Kleinian groups, and the cofinite Hopf property}, Michigan Math. J. {\bf 72} (2022), special issue in honor of Gopal Prasad, 25--49.

\bibitem{BRS} M.R.~Bridson,  A.~W.~Reid and R.~Spitler, 
{\em Absolute profinite rigidity, direct products, and finite presentability},
Duke Math.~J., to appear, {\tt{arXiv:2312.06058}}.


\bibitem{BRW}
M.R.~Bridson,  A.W.~Reid and H.~Wilton,  {\em Profinite rigidity and surface bundles over the circle},
Bull.  London Math.  Soc.  {\bf 49} (2017),  831--841.


  
  
\bibitem{BW}  M.R.~Bridson and H.~Wilton, 
{\em The triviality problem for profinite completions},
Invent. Math. {\bf 202} (2015), 839--874. 
  
  
\bibitem{BW:closed} M.R.~Bridson and H.~Wilton, 
{\em Subgroup separability in residually free groups},
Math. Z.  {\bf{260}} (2008),  25--30.

\bibitem{BW:ggd} M.R.~Bridson and H.~Wilton, 
{\em On the difficulty of presenting finitely presentable groups},
Groups Geom. Dyn. {\bf 5} (2011), no. 2, 301--325.



\bibitem{BW:iso} M.R.~Bridson and H.~Wilton, 
{\em The isomorphism problem for profinite completions of finitely presented,  residually finite groups},
Groups Geom. Dyn. {\bf 8} (2014),  733--745.



\bibitem{brown}  K.S.~Brown, ``Cohomology of groups",  Graduate Texts in Mathematics 87, Springer-Verlag, Berlin-Heidelberg-New York (1982).

\bibitem{owen} O.~Cotton-Barratt, 
{\em Detecting ends of residually finite groups in profinite completions.},
Math.~Proc.~Camb.~Phil.~Soc.  {\bf{155 }}(2013),  379--389.

\bibitem{CW} T. Cheetham-West, {\em Absolute profinite rigidity of some closed fibred hyperbolic 3-manifolds}, Math. Res.~Lett.~{\bf 31} (2024), 615--638.


\bibitem{tam} T.~Cheetham-West, A.~Lubotzky, A.W.~Reid, R.~Spitler,
{\em Property FA is not a profinite property}, Groups Geom. Dyn.~(online first),  DOI:10.4171/GGD/802.

\bibitem{dehn}
M.~Dehn.
\newblock {\em \"{U}ber unendliche diskontinuierliche {G}ruppen},
\newblock  Math. Ann.,  {\bf{71}} (1911),  116--144.
 

\bibitem{davis} M.W.~Davis,  ``Infinite Group Actions on Polyhedra",  
Ergeb.  der Math., vol 77, Springer Verlag, Berlin, 2024.

\bibitem{same-quots}  J.D.~Dixon, E.W.~Formanek, J.C.~Poland and L.~Ribes,  
{\em Profinite completions and isomorphic finite quotients}, 
J.~Pure Appl.~Algebra {\bf 23} (1982), 227--231.
 
 
\bibitem{fruchterMorales}
 J.~Fruchter and I.~Morales,
 {\em Virtual homology of limit groups and profinite rigidity of direct
products,} Israel J.~Math.  (to appear),  {\tt{arXiv:2209.14925}}.

\bibitem{funar} L.~Funar, {\em Torus bundles not distinguished by TQFT invariants},
Geom.~Topol.~{\bf 17} (2013), 2289--2344.
 
 
 \bibitem{weeks} D.~Gabai, G. R.~Meyerhoff and P.~Milley, 
  {\em Minimum volume cusped hyperbolic three-manifolds, }
 J.~Amer.~Math.~Soc. {\bf{22}} (2009), 1157--1215.
  

\bibitem{gromov} M. Gromov, Hyperbolic groups, in {\em Essays in Group Theory} (S.M. Gersten, ed.),
MSRI vol. 8, pp. 75--263. Springer-Verlag, New York, 1987.

\bibitem{gromov2}
M. Gromov, 
{\em Asymptotic invariants of infinite groups,} Geometric Group Theory,  Vol. 2
 (Sussex, 1991),  LMS Lect. Notes.,  vol. 182.  CUP, Cambridge 1993.
 
 \bibitem{gross}
 E.K.~Grossman,
{\em On the residual finiteness of certain mapping class groups}, 
J.~London Math.~Soc. {\bf{2}} (1974/75),  160--164.

\bibitem{groth} A. Grothendieck,
{\em Repr\'{e}sentations lin\'{e}aires et compactification profinie des groupes discrets},
Manuscripta Math.  {\bf{2}} (1970),  375--396.
 
\bibitem{grun}
F. Grunewald, {\em On some groups which cannot be finitely presented}, Jour. 
London Math. Soc.~{\bf 17} (1978), 427--436.


\bibitem{GJZ} F.J.~Grunewald, A.~Jaikin-Zapirain, and P.A.~Zalesskii, {\em Cohomological goodness and the profinite completion of Bianchi groups},  Duke Math. J. {\bf 144} (2008), 53--72.

\bibitem{GPS} 
F.J.~Grunewald, P.F.~Pickel and D.~Segal, {\em Polycyclic groups with isomorphic finite quotients}, Annals of Math. {\bf 111} (1980), 155--195.  

\bibitem{GZ} F.~Grunewald and P.A.~Zalesskii, {\em Genus for groups}, J. Algebra {\bf 326} (2011), 130--168.


 

\bibitem{HW} F. Haglund and D. Wise, {\em Special cube complexes}, 
Geom. Funct. Anal.  {\bf 17}  (2008),  1551--1620. 


 
\bibitem{hempel:rf} J. Hempel, {\em Residual finiteness for 3-manifolds},  in
Combinatorial Group Theory and Topology (Alta, Utah, 1984 ),  S.M.~Gersten and J.R.~Stallings
(eds.),   pp.379--396,
Ann. of Math. Stud., no. 111, Princeton Univ.~Press,   NJ, 1987.

 
\bibitem{hempel} J. Hempel, {\em Some 3-manifold groups with the same 
finite quotients},  {\tt{arXiv:1409.3509}}.



\bibitem{hig}
G. Higman, {\em Subgroups of finitely presented groups},
Proc.  Royal Soc.  Series A {\bf 262} (1961),  455--475.

 
\bibitem{higman}
G. Higman, {\em A finitely generated infinite simple group},
J. London Math Soc. {\bf 26} (1951), 61--64.

 
 

\bibitem{HK}
S.~Hughes and M.~Kudlinska, 
{\em On profinite rigidity amongst free-by-cyclic groups I: the generic case},
{\tt{arXiv:2303.16834}}.

\bibitem{kahler-gang}
S.~Hughes, C.~Llosa Isenrich, P.~Py, M.~Stover, S.~Vidussi,
{\em Profinite rigidity of K\"{a}hler groups: Riemann surfaces and subdirect products},
{\tt{arXiv:2501.13761}}.

\bibitem{Jaikin-Zapirain:parafree} A.~Jaikin-Zapirain,  {\em The finite and solvable genus of finitely generated free and surface
groups},  Res. Math. Sci. {\bf{10}} (2023),  art.~44. 
 

\bibitem{jaikin:fib} A.~Jaikin-Zapirain,  {\em Recognition of being fibered for compact 
3–manifolds},  
Geom.~Topol. {\bf{24}} (2020),  409--420.
 
\bibitem{JL}
A.~Jaikin-Zapirain and A.~Lubotzky, 
{\em Some remarks on Grothendieck pairs},
Groups, Geom.~Dyn. ~(to appear),  {\tt{arXiv:2401.02229}}.



\bibitem{KaS}
H.~Kammeyer and R.~Sauer, 
{\em S-arithmetic spinor groups with the same finite quotients
and distinct $\ell^2$-cohomology},  Groups Geom. Dyn.  {\bf{14}} (2020),  857--869.

\bibitem{KN} M.~Kassabov and N.~Nikolov,  {\em{Profinite properties and property Tau}},
Oberwolfach Report No. 28/2008,  p.~1550.

\bibitem{kielak} D.~Kielak, 
{\em Residually finite rationally solvable groups and virtual fibring,}
J. Amer. Math. Soc.  {\bf{33}} (2020),  451--486.

\bibitem{KS} S.~Kionke and E.~Schesler,
{\em Amenability and profinite completions of finitely generated groups}, Groups Geom.~Dyn.  {\bf{17}} (2023), 
1235--1258.

\bibitem{lack} M.~Lackenby,
{\em Detecting large groups.},  J. Alg.  {\bf{324}} (2010),  2636--2657.

\bibitem{liu} Y.~Liu,  {\em Finite volume hyperbolic $3$-manifolds are almost determined by their finite quotients}, 
Invent.~Math. {\bf231} (2023), 741--804.

\bibitem{liu-vol} Y.~Liu,  {\em Finite quotients,  arithmetic Invariants, and hyperbolic volume},
Peking Math J.~{\bf{8}} (2025), 143--189.

\bibitem{LR} D.D.~Long and A.W.~Reid, {\em Grothendieck's problem for 3-manifold groups}, Groups Geom.~Dyn. {\bf 5} (2011),
479--499.

\bibitem{LoLu} J. ~Lott and  W. ~L\"uck, {\em $L^2$-topological invariants of 
3-manifolds}, Invent. Math. {\bf 120} (1995), 15--60. 


\bibitem{lub-FPn} A.~Lubotzky,
{\em Finiteness properties and profinite completions},  
Bull. Lond. Math. Soc. {\bf{46}} (2014), 103--110. 
 
 
\bibitem{LubSeg} A.~Lubotzky and D.~Segal,
``Subgroup Growth",  Prog. Math., vol 212, Birkh\"{a}user Verlag,  Basel,  2003. 

 

\bibitem{luck:gafa} W.~L\"{u}ck, {\em  Approximating $L^2$-invariants by their finite-dimensional analogues},
Geom. Funct. Anal.    {\bf{4}} (1994),  455--481.


\bibitem{MR} C.~Maclachlan and A.~W.~Reid,  ``The Arithmetic of Hyperbolic
3-Manifolds", Graduate Texts in Math. , vol.~219, Springer-Verlag, Berlin, 2003.



\bibitem{malcev} A. Malcev,
{\em On isomorphic matrix representations of infinite groups}, 
Rec~Math.[Mat.~Sbornik] N.S.,  Ser.{\bf{8}}.50 (1940), 405--422.

\bibitem{meskin} S.~Meskin,  
{\em A Finitely generated residually finite group with an unsolvable word problem},
Proc.   Amer.   Math.   Soc.   {\bf 43} (1974),   8--10.  

 

\bibitem{mihailova}
K.A. Mihailova, {\em The occurrence problem for direct products of groups},
Dokl. Akad. Nauk SSSR {\bf 119} (1958),  1103--1105.

\bibitem{miller} C.F.~Miller~III, 
``On group-theoretic decision problems and their classification",
Ann. Math. Studies,
No. 68, Princeton University Press, NJ, 1971.


\bibitem{nek}
V.~Nekrashevych,
{\em An uncountable family of 3-generated groups with isomorphic profinite completions,} 
Internat.~J.~Algebra Comput.  {\bf{24}} (2014), 33--46.

\bibitem{NS} N.~Nikolov and D.~Segal, {\em On finitely generated profinite 
groups.~I. Strong completeness and uniform bounds},  Annals of Math.~{\bf 165}
(2007), 171--238. 

\bibitem{probs} G.A.~Noskov, V.N.~Remeslennikov and V.A.~Romankov,
{\em{Infinite Groups}} (Russian).  Algebra, Geometry, Topology  vol 17 (1979), Akad. Nauk SSSR,
Vsesoyuz. Inst. Nauchn. i Tekn. Informatsii,  Moscow,  pp. 65--157. 


\bibitem{novikov} P.S. Novikov, {\em 
On the algorithmic unsolvability of the word problem in
group theory}, Trudy Mat. Inst. Steklov {\bf 44} (1955) 1--143.



\bibitem{Oll}
Y. ~Ollivier, {\em Some small cancellation properties of random groups}, 
Intl. J. Alg. Comput.  {\bf{17}} (2007),  37--51.

\bibitem{OS} 
A.Yu. Ol’shanskii and  M.V. Sapir,
{\em A 2-generated,  2-related group with no non-trivial finite quotients},
Alg.~Discr.~Math.  {\bf 2} (2007),  111--114.


\bibitem{pickel}  P.F. ~Pickel, {\em Finitely generated nilpotent groups with isomorphic quotients}, 
Trans.~Amer.~Math.~Soc. {\bf 160} (1971), 327--341.


 \bibitem{piwek} P.~Piwek, {\em Profinite rigidity properties of central extensions of $2$-orbifold groups}, 
 Alg.~Geom.~Topol. (to appear), {\tt{arXiv.2304.01105}}.
   
\bibitem{PT} V.P.~Platonov and O.I.~Tavgen, {\em Grothendieck's problem on profinite completions and representations of groups}, K-Theory {\bf 4} (1990), 89--101.

\bibitem{pyber}
L.~Pyber, 
{\em Groups of intermediate subgroup growth and a problem of Grothendieck,} Duke
Math.~J.  {\bf{121}} (2004), 169--188.

\bibitem{PW} P. Przytycki and D. T. Wise,  {\em Separability of embedded surfaces in 3–manifolds}, 
Compos.~Math. , {\bf{50}} (2014),  1623--1630.

\bibitem{PW1} P.~Przytycki and D.T.~Wise, {\em Mixed 3-manifolds are virtually special}, 
J.~Amer.~Math.~Soc.,  {\bf{31}} (2018),  319--347.


\bibitem{rabin} M.O. ~Rabin, {\em Recursive unsolvability of group theoretic problems},  Annals of Math {\bf{67}} (1958), 172--194.



\bibitem{reid:stA} A.W.~Reid, \emph{Profinite properties of discrete groups}, in Groups St Andrews 2013, LMS Lecture Note Ser. {\bf 422},  pp.~73--104,  Cambridge Univ. Press (2015). 


\bibitem{reid:icm} A.W.~Reid, \emph{Profinite rigidity},  Proc. ICM, Rio de Janeiro $2018$, 
Vol 1,  1193--1216, World Sci. Publ., Hackensack, NJ, 2018.


\bibitem{reid:kias} A.W.~Reid, \emph{Profinite rigidity, Grothendieck pairs and
low-dimensional topology},  KIAS  Lecture Note Series, to appear.

\bibitem{remes} V.N.~Remeslennikov,
{\em Conjugacy of subgroups in nilpotent groups,} Algebra i Logika Sem.
{\bf 6} (1967),  61--76. (Russian)

\bibitem{remy} B.~R\'{e}my  
{\em G\'{e}om\'{e}trie des groupes et compl\'{e}tion profinie,  d'apr\`{e}s Martin Bridson,   Alan Reid {et al.}},
S\'{e}minaire Bourbaki, expos\'{e}e 1221,  Ast\'{e}rique, vol.~454.  Soc.~Math.~France, Paris 2024, pp.~429--465. 


\bibitem{RZ} L.~Ribes and P.A.~Zalesskii,  ``Profinite Groups", 
Ergeb.~der Math.~{\bf 40}, Springer-Verlag, Berlin, 2000.
 
 
\bibitem{rips}
E. Rips,  {\em Subgroups of small cancellation groups},
  Bull. London Math Soc. {\bf 14} (1982), 45--47. 
  
\bibitem{sageev} 
M.~Sageev, 
{\em CAT$(0)$ cube complexes and groups},
in ``Geometric group theory", 7--54, IAS/Park City Math. Ser., 21, Amer. Math. Soc., Providence, RI, 2014.

  
\bibitem{polyC} G.~Sabbagh and J.S.~Wilson, 
{\em Polycyclic groups, finite images, and elementary equivalence, }
Archiv der Math.  {\bf{57}} (1991),  221--227.
  
\bibitem{schwer} P.~Schwer,
``CAT(0) Cube Complexes:
An Introduction",    Lect. Notes Math.,  vol.~2324,  Springer Verlag, Berlin, 2023.

  
\bibitem{scott} G.P.~Scott,  {\em The geometries of $3$-manifolds}, Bull. London Math. Soc. {\bf 15} (1983), 401--487.

\bibitem{serre-trees}  J-P.~Serre,  ``Trees", Springer Verlag, Berlin, 1980.


\bibitem{serre:CG} J-P. Serre,  ``Cohomologie Galoisienne", Fifth edition. Lect. Notes 
Math., vol.~5, Springer-Verlag, Berlin, 1994.


\bibitem{serre1964} J-P. Serre,  {\em Exemples de vari{\'e}t{\'e}s projectives conjugu{\'e}es non hom{\'e}omorphes}, 
C.R. Acad. Sci., {\bf{258}} (1964),  4194--4196.
 
 
\bibitem{serre1974} J-P. Serre,  {\em Problems},  in Proc. ~Conf. ~Canberra
(1973),  Lecture Notes in Math., vol 72,  pp. ~734--735, Springer Verlag, Berlin,  1974.
 
\bibitem{slobo}
A.M.~Slobodsko{\u\i}.
\newblock Undecidability of the universal theory of finite groups.
\newblock {\em Algebra i Logika}, 20.2 (1981),   207--230.



\bibitem{spitler} R.~Spitler, {\em Profinite Completions and Representations of Finitely Generated Groups}, PhD thesis, Purdue University, 2019.
 
\bibitem{stall} J.R.~Stallings,
\newblock{\em A finitely presented group whose 3-dimensional integral homology is not finitely generated},
\newblock Amer. J. Math. {\bf 85} (1963), 541--543.


\bibitem{stall-graphs} J.R.~Stallings,  {\em Topology of finite graphs}, 
Invent. Math. {\bf{71}} (1983),  551--565.



\bibitem{stebe}
P. ~Stebe,
{\em Conjugacy separability of groups of integer matrices},
Proc. Amer. Math. Soc. {\bf{32}} (1972), 1--7.

\bibitem{stover1} 
M.~Stover,
\emph{Lattices in ${\rm{PU}}(n,1)$ that are not profinitely rigid},
Proc. Amer. Math. Soc. {\bf{147}} (2019), 5055--5062.


\bibitem{stover2} 
M.~Stover, {\em Algebraic fundamental groups of fake projective planes},
J.~Eur.~Math.~Soc. ~(Online first, 2024),  DOI 10.4171/JEMS/1536.

\bibitem{Sun} H.~Sun, {\em All finitely generated $3$-manifold groups are Grothendieck rigid},  
Groups Geom.~Dyn.  {\bf{17}} (2023),  {385--402}.

\bibitem{wilkes:sfs} G. Wilkes, {\em Profinite rigidity for Seifert fibre spaces}, Geom. Dedicata {\bf 188} (2017), 141--163.
 
\bibitem{wilkes-jsj} G. Wilkes, {\em Profinite rigidity of graph manifolds and JSJ decompositions of 3-manifolds}, 
J.~Alg.  {\bf{508}}(2018), 538--587.

\bibitem{wilkesJPAA} G. Wilkes, {\em Relative cohomology theory of profinite groups},
J.~Pure.~Appl.~Alg.  {\bf{223}} (2019),   1617--1688.


\bibitem{wilkesNZ} G. Wilkes, {\em Profinite completions, cohomology and JSJ decompositions of compact 3-manifolds},  New Zealand J.~Math.  {\bf{48}} (2018),  101--113.



\bibitem{wilkes-book} G. Wilkes,  ``Profinite groups and residual finiteness",  EMS Press, Berlin, 2024.
 

\bibitem{W:if} H. Wilton, {\em The congruence subgroup property for mapping class groups and the residual finiteness of hyperbolic groups},  with an appendix by A.~Sisto, {\tt{arXiv:2410.00556}}. 

\bibitem{WZ1} H. Wilton and P. A. Zalesskii, {\em Distinguishing
geometries using finite quotients},  
Geom. ~Topol. {\bf{21}} (2017), 345--384.
 

\bibitem{WZ3} H. Wilton and P. A. Zalesskii, {\em Profinite detection of 3-manifold decompositions}, 
Compos.~Math.  {\bf{155}} (2019),  246--259.

\bibitem{wise:qc}
D.~T.~Wise,  ``The structure of groups with a quasiconvex hierarchy",  
Ann. Math. Studies,
No. 209, Princeton University Press,  NJ,  2021.  
 
\bibitem{wise-invent}
D.T.~Wise,
{\em The residual finiteness of negatively curved polygons of finite groups},
Invent. Math.  {\bf{149}} (2002),  579--617.
 
\bibitem{wise:omni}
D.T.~Wise,
{\em Subgroup separability of graphs of free groups with cyclic edge
  groups},  Quart.~J.~Math {\bf{51}} (2000),  107--129.


\bibitem{wise:rf}
D. Wise, \emph{A residually finite version of Rips's
construction},
Bull. London Math. Soc. {\bf 35} (2003), 23-29.

\bibitem{wise-sc} D.~T.~Wise,
\newblock {\em Cubulating small cancellation groups},
\newblock GAFA, {\bf 14} (2004), 150--214. 
 

\bibitem{xu}
Xiaoyu Xu,
{\em  Profinite almost rigidity in 3-manifolds},  {\tt{aXiv:2410.16002}}.




\end{thebibliography}
\end{document}